\documentclass[12pt]{amsart}
\usepackage[foot]{amsaddr}		
\usepackage{enumerate}
\usepackage{geometry}%margin setting
\usepackage[shortlabels]{enumitem}
\usepackage[english]{babel}
\usepackage[utf8]{inputenc}

\newcommand{\add}[1]{{#1}}

\newcommand{\R}{\mathbb{R}}
\usepackage{verbatim,enumerate} % commenting multi-lines
\newtheorem{remark}{Remark}
\newtheorem{assumption}{Assumption}
\newtheorem{theorem}{Theorem}[section]

\newtheorem{lemma}[theorem]{Lemma}
\newtheorem{definition}{Definition}[section]
\newtheorem{algorithm}{Algorithm}[section]
\newtheorem{example}{Example}[section]
\usepackage{fancyhdr}
	\usepackage{amsfonts}
	\usepackage{amssymb}
	\usepackage{amsmath}
	\usepackage{commath}
	\usepackage{mathtools}

\usepackage{multicol}
\usepackage{booktabs}
\usepackage{multirow}
\usepackage{array}
\newcolumntype{L}[1]{>{\raggedright\let\newline\\\arraybackslash\hspace{0pt}}m{#1}}
\newcolumntype{C}[1]{>{\centering\let\newline\\\arraybackslash\hspace{0pt}}m{#1}}
\newcolumntype{R}[1]{>{\raggedleft\let\newline\\\arraybackslash\hspace{0pt}}m{#1}}
\newcommand{\ba}{\begin{eqnarray}}
\newcommand{\ea}{\end{eqnarray}}
\newcommand{\bas}{\begin{eqnarray*}}
\newcommand{\eas}{\end{eqnarray*}}
\newcommand{\porescale}{porescale}

\usepackage{graphicx}
\usepackage[rightcaption]{sidecap}
	
\begin{document}
   \title{Numerical Analysis of a Parabolic Variational Inequality System  Modeling Biofilm Growth at the Porescale}
\author{Azhar Alhammali \and Malgorzata Peszynska}
\address{Department of Mathematics, Oregon State University, Corvallis, OR 97330
\\
email: 
\texttt{alhammaz@oregonstate.edu}
\\
\texttt{mpesz@math.oregonstate.edu}
}
%\date{\today}
\maketitle
%%%%%%%%%%%%%%%%%%%%%%%%%%%%%%%%%%%%%%%%%%%%%%%%%%%%%%%%%%%%%%%%%%%%%
%%-----------------------Introduction------------------------------
\section*{Abstract}

In this paper we consider a system of two coupled nonlinear diffusion--reaction partial differential equations (PDEs)
which model the growth of biofilm and \add{consumption} of the nutrient. At the scale of interest the biofilm density is subject to a pointwise constraint, thus the biofilm PDE is framed as a parabolic variational inequality. We derive rigorous error estimates for a finite element (FE) approximation to the coupled nonlinear system and confirm experimentally that the numerical approximation converges at the predicted rate.  We also show simulations in which we track the free boundary in the domains which resemble the pore scale geometry and in which we test the different modeling assumptions. 

\section*{Keywords} Parabolic variational inequality, coupled system, biofilm growth, finite elements, error estimates, semismooth Newton solver.

\section{Introduction}\label{sec:int}
Coupled bio-chemical interactions involving microbial cells are very important in many applications including the food industry and medicine. In this paper we focus on the applications involving ``selective plugging'' by biofilm in porous media such as, e.g., in microbial enhanced oil recovery (MEOR) and in carbon sequestration.   In these applications the microbes coat the grains of the porous medium and plug the high-permeability zones thereby enhancing overall recovery \cite{MacLeod88} or preventing leakage by promoting mineralization \cite{ebigbo2012,verba2016,CVTAKPT}.

The mathematical model we consider was proposed in \cite{PTISW2016} to describe some microorganisms and nutrient suspended in ambient fluid. Their densities (or concentrations) are denoted by $B(x,t)$ and $N(x,t)$, respectively, and these suspensions spread by diffusion as in
\begin{subequations}
\label{eq:ssystem}
\ba
\label{eq:systemB}
           \frac{\partial B}{\partial t} - \nabla \cdot (D_B \;\nabla B) -\Lambda &= F(B,N)+f(x,t),
           \\
           \label{eq:systemN}
 \frac{\partial N}{\partial t} - \nabla \cdot (D_N \;\nabla N) &= G(B,N)+g(x,t).
\ea
\end{subequations}
The nonnegative diffusivities $D_B,D_N$ are discussed in the sequel, and their choice plays an important role in the dynamics of spreading. For the 
growth $F$ and consumption $G$ a common choice are Monod functions
\ba
\label{eq:monod}
F(B,N)=\kappa_B P(N)B,\;  G(B,N)=-\kappa_N P(N)B,\;\;
P(N)=\frac{N}{N+N_0},
\ea
where $k_B\geq k_N>0, N_0>0$ are known constants. The external source and sink terms are denoted by $f,g$. The model is completed by boundary and initial conditions specified later. \add{Here we assume that the fluid is at rest, so the model \eqref{eq:ssystem} does not include transport by advection}. \add{See references to the biofilm--nutrient models in, e.g., \cite{ebigbo2012,Eberal,PTISW2016}.} 

The term $\Lambda$ in \eqref{eq:systemB} enforces the constraint $B(x,t) \leq B^*$ written as Nonlinear Complementarity Condition (NCC)
\ba
\label{eq:constraint}
B(x,t) \leq B^*, \;\; \Lambda \leq 0, \;\; \Lambda(B-B^*)=0.
\ea
This salient feature of the model in \cite{PTISW2016} leads to the main challenges addressed in this paper. We explain the need for \eqref{eq:constraint}, and the role and the relationship between $\Lambda$, $B$, and $B^*$ in what follows. 

Without the constraint \eqref{eq:constraint} is not active, i.e., with $\Lambda=0$ in \eqref{eq:ssystem}, the growth of the solution $B(x,t)$ predicted by \eqref{eq:ssystem} under some conditions such as no-flux boundary conditions and with abundant nutrient can be exponential. This is only realistic at large spatial scales of m or km, typical for laboratory containers and reservoirs, or at small concentrations.    
\add{Other features of \eqref{eq:ssystem}, e.g., of dynamics of $B+\tfrac{\kappa_B}{\kappa_N}N$ under various assumptions on $f$ and $g$ can be analyzed, but are out of our scope. } 
%

%\medskip
\add{We are interested in the  growth of microbes at the mm scale relevant, e.g., to biomedical applications, or in the {\em biofilm phase} at the pore-scale of $\mu$m.  At that scale the constraint \eqref{eq:constraint} comes from the fact that the microbial cells have a finite size, so their density is limited.  
In addition, when $B\approx B^*$ is reached, the cells form  \add{the biofilm, a gel-like substance made of microbial cells which are surrounded by an extracellular polymer substance (EPS) matrix. The EPS matrix protects the cells, which tend to aggregate, and it is not penetrated by larger molecules, e.g., of imaging contrast. The interface between the biofilm domain denoted by $\Omega^*(t)=\{x \in \Omega: B(x,t)=B^*\}$, and the surrounding fluid domain 
 $\Omega^-(t)=\{x: B(x,t)<B^*\}$  is a {\em  free boundary} $\Gamma(t) =\partial \Omega^*(t) \cap \partial \Omega^-(t)$ visible to human eye and easily recognized by x-ray $\mu$-CT tomography or other microscopy equipment; see \cite{PTISW2016,verba2016,CVTAKPT}. The growth of $\Omega^*(t)$ proceeds through interface only; this explains the tapering of exponential growth reported, e.g., in \cite{CVTAKPT}.} 
Finding  $\Lambda \neq 0$ and the free boundary $\Gamma(t)$ is part of the evolution problem, and the PDE \eqref{eq:systemB} is formally a nonlinear {\em parabolic variational inequality} (PVI) coupled to the nutrient dynamics.}

\medskip
{\bf Theoretical and practical challenges and contributions of this paper}.  
First, the solutions to  free-boundary problems and variational inequalities such as \eqref{eq:ssystem} have low regularity. For this reason we consider only a linear Galerkin Finite Element approximation coupled with fully implicit time-stepping for which we prove convergence roughly of $O(h)$ in the $L^2$ norm if $\Delta t=O(h)$. The main theoretical challenges are the coupled nonlinear nature of the system and the constraint. 
We confirm the rate with numerical experiments and describe a robust and efficient nonlinear solver. Second, we work in complicated geometries obtained from imaging, such as the void space between the grains of a porous medium at the {\porescale}, and study some model variants which support qualitative behavior in experiments \cite{verba2016,CVTAKPT,PTISW2016}. \add{The computational challenge is to carry out the simulations in complicated domains which require careful grid generation and pre- and post- processing.}

\medskip
{{\bf Literature remarks.} Numerical simulation of biofilm growth is of interest because the physical experiments with biofilm and its imaging are quite difficult \cite{PTISW2016,ebigbo2012,verba2016,CVTAKPT}. Various models for biomass growth and biofilm growth address the growth at the {\porescale} and the permeability of the porous domains plugged up by biofilm; however, their focus is selective.  In particular, the models \cite{ebigbo2012,schulz2017} do not impose the constraint \eqref{eq:constraint}, while the model in \cite{pop2010} considers a simplified free boundary geometry in a strip.   In turn, the hybrid differential--lattice models in \cite{Eberal} and \cite{Valocchi} enforce \eqref{eq:constraint}, but their structure does not permit numerical convergence analysis. A complex family of phase field models is considered in \cite{zhangklapper}, but its computational complexity seems prohibitive at the {\porescale}. Further considerations of complexity and uncertainty in upscaling are in \cite{CKP18}.
Finally, the model in \cite{PTISW2016} is more general than \eqref{eq:ssystem}, since it considers advective transport, coupled Navier-Stokes flow solver, as well as dynamic upscaling of the flow solutions to obtain permeability. However, there is no numerical analysis in \cite{PTISW2016}. }

The literature  on the topic of numerical approximation of parabolic systems is enormous; we recall only a few results which directly guide our work.  The classical results by Wheeler \cite{Wheeler73} for unconstrained scalar problems  (see also \cite{Thomee1997}, Chapter 13) establish second order $l^\infty(L^2)$ error estimate of the linear finite approximation in space and backward Euler in time, but require high regularity of the solution {$B$}, such as {$\frac{\partial^2 B}{\partial t^2} \in L^2(Q)$}. 
The works on elliptic and parabolic variational inequalities (EVI and PVI), such as on first order convergence of Galerkin FE \cite{BHR1977,Elliott1981,Baiocchi1989},  account for the lack of regularity of their solutions but handle only scalar linear problems. 
The analyses of Johnson \cite{Johnson1976} for the subproblem \eqref{eq:systemB} are the closest to what we consider here but they require $F=0$ and $D_B=$const. The paper \cite{Vuik1990}  considers more general finite differencing in time, but requires stronger regularity to obtain first order of convergence in {$l^\infty(L^2)$.} For systems without constraints,  the analysis in \cite{BD2007} allows degeneracy in {the diffusivity in $D_B(B)$} and gives  $O(h^{1/2})$ order of convergence in {$l^\infty(L^2)$}-norm.  Finally there is more work on FE approximation to free boundary problems \cite{Verdi, Nochetto, Rulla}; many recent papers focus on {adaptivity} which we do not consider here. 

\medskip
{\bf Plan of the paper.} In Sec.~\ref{sec:model} we provide the details on the model. In Sec.~\ref{sec:error} we define the FE approximation and prove our estimate. In Sec.~\ref{sec:solver} we describe the solution method framed as a semi-smooth Newton solver for the NCC. In Sec.~\ref{sec:NE} we present  numerical experiments which confirm the predicted convergence rate  and illustrate different qualitative behavior depending on the assumed nonlinear diffusivity models.

%%
%%
%%%------------------------------------------------------------------
\section{Background and formal setting}
\label{sec:model}

\add{In this section we provide details of our model which is a parabolic system of PDEs under a constraint, i.e., of parabolic variational inequalities. We discuss the rigorous setting including the PDE model in the sense of distributions and its weak formulation. }

Let $\Omega$ be an open bounded domain in {$\mathbb{R}^d$; $d$=$2$, with  a smooth boundary $\partial \Omega$.} {(In our numerical experiments cover we consider $d=1,2,3$.)} Throughout the paper, we use standard notation on $L^2(\Omega)$ and Sobolev spaces $H^2(\Omega)$ and $H_0^1(\Omega)$; see, e.g., \cite{Showater97}. Let $\|\cdot\|_0$ denote the norm  on $L^2(\Omega)$, and $\|\cdot\|_s$ the norm  on $H^s(\Omega)$ where $s$ is a nonnegative integer. $J=[0,T]$ denotes the time interval with final time $T>0$. We write $u\in L^2(H_0^1)$ to mean $u \in L^2(J;H_0^1(\Omega))$; similar shorthand is applied to other notation on functional spaces. In particular, the notation $L^2(L^2)$ means $L^2(J;L^2(\Omega)) =L^2(Q)$ with $Q=\Omega \times J$.   

First we discuss the constraint \eqref{eq:constraint}, and denote $\R^*={(-\infty ,B^*]}$. The indicator function $I_{\R^*}$ takes value $0$ in the set $R^*$ and $\infty$ outside: 
$
I_{\R^*}(B)=\left\{
\begin{array}{ll}
0 &  B \in \R^*,\\
+\infty& \mbox{otherwise}.
\end{array}
\right. 
$
Denote by $\partial I_{\R^*}$ the subgradient of $I_{\R^*}$, the multivalued constraint graph 
\[
\partial I_{\R^*}=\left((-\infty,B^*)\times \{0\}\right) \cup \left(\{B^*\} \times [0,\infty)\right).\]

\add{Now we write the model \eqref{eq:ssystem} formally in the sense of distributions using the operator $\partial I_{\R^*}(B)$ instead of $-\Lambda$. Formally, when $\partial I_{\R^*}(B)$ appears in an equation, $B$ must 
satisfy \eqref{eq:constraint}, i.e., 
be in the domain $K$ of $\partial I_{\R^*}(B)$, where the set $K:=\{B \in H_0^1(\Omega) ; \; B \le {B^\ast} \;\; {a.e} \;\;\mbox{on}\;\; \Omega \}$ is a convex subset of $H_0^1(\Omega)$. The model  reads}
\begin{subequations}
\label{eq:pvi}
\ba
\label{eq:pviB}
\frac{\partial B}{\partial t} - \nabla \cdot (D_B \;\nabla B) + \partial I_{(-\infty ,B^*]}(B)  &\ni F(B,N) +f(x,t),& \;\; \mathrm{a.e.\ in\ } \Omega, \; t>0, 
\\
\label{eq:pviN}
\frac{\partial N}{\partial t} - \nabla \cdot (D_N \;\nabla N) &= G(B,N)+g(x,t),&  \mathrm{a.e.\ in\ } \Omega, \; t>0.
\ea
\add{The operator $\partial I_{\R^*}(B)$ is multivalued, so we use the symbol $\ni$. However, there is no ambiguity in the choice of the particular selection $\Lambda \in \partial I_{\R^*}(B)$, since this selection is always exactly the value that keeps $B(x,t)$ in the convex set $K$.}  

The model is completed with the initial and boundary conditions as follows
\ba
B(x,0)&= B_{init}(x)&\;\; x \in \Omega, \label{eq:pv3}\\
N(x,0) &= N_{init}(x)&\;\; x \in \Omega, \\      
B(s,t)&=0,&\;\;s \in \partial \Omega, \; t>0, \label{eq:pvbc1}
\\ N(s,t)&=0,& \;\; s \in \partial \Omega, \; t>0. \label{eq:pvbc2}
\ea
\end{subequations}
\begin{remark}
The choice of Dirichlet conditions \eqref{eq:pvbc1}--\eqref{eq:pvbc2}
is merely for the easiness of analysis. In practice, Neumann conditions are realistic in simulations of a closed system. Furthermore, while our analysis is restricted to the case where $D_B,D_N$ do not depend on $B,N$, we consider other variants in numerical experiments. 
\end{remark}

%%------------------------------------------------------
%\subsubsection{}
The weak form of \eqref{eq:pvi} is the variational inequality
%%%%%%
\begin{subequations}
\label{eq:PVI}
	\begin{eqnarray}
B(t) \in K: \left(  \frac{\partial B}{\partial t}, \psi-B \right) +(D_B\nabla  B,\nabla(\psi -B))&\ge & ( F(B,N){+f}, \psi -B),  \forall \psi \in K, \label{eq:cd}\\
N(t) \in H_0^1(\Omega): \left(\frac{\partial N}{\partial t}, \chi\right)+(D_N \nabla N, \nabla \chi)&=& (G(B, N)+g, \chi),  \forall \chi \in H_0^1(\Omega), \label{eq:nv}\\
B(x,0)&=&{B_{init}(x)}, \label{eq:dd}\\
 N(x,0)&=&{N_{init}(x)}, \label{eq:de}
 \end{eqnarray}
\end{subequations}
where 
$( \cdot, \cdot)$ is the duality pairing; (also, the scalar product on $L^2(\Omega)$); \add{see \cite{brezis,Showater97,Ulbrich2011}}. \add{The symbol $\geq$ in \eqref{eq:cd} expresses the fact that the admissible solutions and test functions are defined by inequalities, thus they constitute a convex set $K$ rather than the space $H_0^1(\Omega)$. In the numerical model the symbol ``$\geq$'' is replaced by ``$=$'' and the presence of a Lagrange multiplier.  }

For the needs of subsequent numerical analysis model  we make some assumptions on the data and the regularity of the solutions typical for parabolic variational inequalities.  
\begin{assumption}
\label{as:data}
\begin{enumerate}[(A)]
    \item  $D_B, D_N:\Omega\to \R$ are Lipschitz continuous with a Lipschitz constant $R$, and $0 < \mu_1\le D_B(x)\le \mu_2, \; \; 0 < \nu_1\le D_N(x)\le \nu_2 \; \; \mbox{for $x\in \Omega$}$ and some $\mu_1,\mu_2,\nu_1,\nu_2$. 
    \label{c:Diff}
\item  $F(B,N)$ and $G(B,N)$ are  smooth functions, Lipschitz with respect to $B$ and to $N$ with a Lipschitz constant $M$. Further assume $F$ and $G$ are linear in $B$ and uniformly bounded in $N$ on  $\mathbb{R}\times \mathbb{R}^+$,    
\ba
\label{eq:zero}
F(B,0)=0=F(0,N), \; G(B,0)=0=G(0,N)\; \forall B,N\in \mathbb{R}.
\ea
[We note that Monod growth functions defined by \eqref{eq:monod}  satisfy these conditions. \label{c:FG}] 
\item $f,g\in C(L^\infty)$, $\frac{\partial f}{\partial t}\in L^2(L^\infty)$.\label{c:fg} 
\item $B^*>0$ is given. \label{c:const}
\item  { $B_{init}, N_{init}\in W^{2,\infty}$, and $B_{init}\le B^*$.}\label{c:init}
\item {$B,N\in L^\infty(W^{2,p})$;  $1\le p  < \infty $, and $\frac{\partial B}{\partial t}, \frac{\partial N}{\partial t}\in L^2(H_0^1)\cap L^\infty(Q)$. Also, $\frac{\partial^2 N}{\partial t^2}\in L^2(Q)$.}\label{c:sol}
\end{enumerate}
\end{assumption}

\begin{remark}
The assumption $\frac{\partial^2 N}{\partial t^2}\in L^2(Q)$ can be easily dropped. This is useful if  the second PDE of system \eqref{eq:pvi} is constrained.
The assumptions \ref{c:sol}  on regularity of $B$ are realistic; in general  $\tfrac{\partial^2 B}{\partial t^2} \not \in L^2(H^{-1})$.
\end{remark}

We denote by $\chi_A$ the characteristic function of a set $A$. We
shall use the elementary inequality 
\begin{equation}
a b \le \frac{\epsilon }{2}a^2+\frac{1}{2 \epsilon}b^2 \; \; \forall a, b \in \mathbb{R}, \; \epsilon >0, \label{eq:eps}
\end{equation}
and Poincaré-Friedrichs inequality 
\begin{equation}
\|\psi\|_0 \le C_{PF} \|\nabla \psi\|_0, \;\; \psi \in H_0^1(\Omega);  \;\; C_{PF}>0, \label {eq:PF}
\end{equation}

%%------------------------------------------------------

%%------------------------------------------------------
\section{The Approximation and Error Estimate}
\label{sec:error}

In this section we formulate a fully discrete approximation to \eqref{eq:PVI} using backward Euler scheme in time and piecewise linear finite element method in space.  The notation is fairly standard; see, e.g., \cite{Thomee1997}. 

%%%%
\subsection{Discrete Problem}
For $h> 0$, {let  $\mathcal{T}_h=\{T_i\}$ be a conforming triangulation of $\Omega$ such that the length of each side of any element $T_i$ is at most $h$. We assume that no vertex of any triangle lies in the interior of another triangle, and all angles of the  triangles are bounded below by a fixed positive constant, and that $\Omega=\bigcup_{T \in \mathcal{T}_h}$.

Define the space $V_h$ for piecewise linear approximations and its convex subset $K_h$
\bas
V_h=\{\psi \in C(\bar \Omega): \psi \; \mbox{ is linear on each } T_i, \; \psi=0 \mbox{ on } \partial\Omega\}, \;\; 
K_h=V_h \cap K.
\eas
{In what follows we work with the product space
 $\mathbf{H}=H_0^1(\Omega)\times H_0^1(\Omega)$, with the norm $\|(B,N)\|_{\mathbf{H}}=\sqrt{\|B\|_1^2+\|N\|_1^2}$. Let $\mathbf{V}_h=V_h\times V_h \subset \mathbf{H}$. }
%%%

We use uniform time-stepping on $J$ 
for simplicity of the exposition. One can also easily use non-uniform or adaptive time-stepping, but we skip details. Let $\Delta t=N_T^{-1} T;  \; N_T $ a positive integer, $t_n =n \Delta t$,  $J_n =(t_n, t_{n+1}]$.  We denote $\psi^n =\psi(t_n)$. In this section $\partial$ refers to a difference $\partial \psi^n =(\psi^{n+1} -\psi^n)/\Delta t$ for any $\psi$.  Let $\Upsilon=\{t_0,\ldots,t_{N_T}\}$ be the set of time steps.

We approximate \eqref{eq:PVI} as follows. We seek 
$B_h:\Upsilon \to K_h$ and  $N_h:\Upsilon \to V_h$ which satisfy, for $n=0, \ldots, N_T-1$ 
%%%%
\begin{subequations}
\label{eq:DPVI}
	\begin{multline}
 (\partial B_h^n, \psi-B_h^{n+1})+(D_B \nabla  B_h^{n+1}, \nabla \psi - \nabla B_h^{n+1})\\
 \ge   (F(B_h^{n+1},N_h^{n+1})+f^{n+1},\psi-B_h^{n+1})  \; \;\forall \psi \in K_h, \label{eq:c} 
 \end{multline}
 \begin{align}
(\partial N_h^n, \chi)+(D_N \nabla N_h^{n+1}, \nabla \chi)& = (G(B_h^{n+1}, N_h^{n+1})+g^{n+1}, \chi)& \forall \chi \in V_h. \label{eq:nnv}\\
\|B_h^0-B_{init}\|_0 &\le Ch, \label{eq:d}\\
B_h^0&=I_hB_{init}\\
N_h^0&= I_hN_{init}.
\end{align}
\end{subequations}
Here $I_h$ is the local smoother defined in the sequel. 

\subsection{Existence and uniqueness of solutions}
Here we prove that the fully discrete problem \eqref{eq:DPVI} has a unique solution under mild assumptions on the size of $\Delta t$. 

\begin{lemma}{Assume that (i) $\gamma=\min\{\mu_1, \nu_1\} >2 MC_{PF}^2$}. If this does not hold, (ii) assume that  $\Delta t$ is sufficiently small, \add{and in particular that $\Delta t < \frac{C_{PF}^2}{2MC_{PF}^2 -\gamma}$}. Then the problem \eqref{eq:DPVI} has a unique solution. 
\end{lemma}
\begin{proof}
%%%%---------
For $n=0,\ldots,N_T-1$, ~(\ref{eq:DPVI}) can be written as 
\begin{subequations}
\begin{align*}
    a_{\Delta t}(N_h^{n+1};B_h^{n+1}, \psi-B_h^{n+1})&\ge (B_h^n+\Delta t f^{n+1}, \psi-B_h^{n+1}) \;\; \forall \psi \in K_h,\\
     b_{\Delta t}(B_h^{n+1};N_h^{n+1},\chi)&=(N_h^n+\Delta t g^{n+1},\chi)\; \forall \chi \in V_h,
\end{align*}
\end{subequations}
where $a_{\Delta t}:V_h\times V_h\longrightarrow \mathbb{R}$, $b_{\Delta t}:V_h\times V_h\longrightarrow \mathbb{R}$ are defined as 
\begin{subequations}
\begin{align*}
    a_{\Delta t}(N;B,\psi)&=(B,\psi)-\Delta t(F(B,N),\psi)+\Delta t(D_B\nabla B,\nabla \psi), \;\; \forall (B,\psi)\in V_h\times V_h,\\
    b_{\Delta t}(B;N, \chi)&=(N,\chi) -\Delta t(G(B,N), \chi)+\Delta t (D_N \nabla N, \nabla \chi), \;\; \forall (N,\chi)\in V_h\times V_h.
\end{align*}
\end{subequations}
In what follows we suppress the indices $h$, and $n$+$1$ in $N_h^{n+1},B_h^{n+1}$. Define $L:\mathbf{H}\longrightarrow \mathbf{H}^\prime$ as 
\[
(L(B,N),(\psi,\chi))=a_{\Delta t}(N;B,\psi)+b_{\Delta t}(B;N,\psi) \;\; \forall (B,N), (\psi,\chi) \in \mathbf{H}.
\]
We will show that $L$ is monotone, continuous, and coercive. From this it follows that $L$ is demicontinuous, and we can apply (\cite{Barbu12} Corollary 2.8, page 73) to prove existence of the solution in $K_h \times V_h$.  If $L$ is also strictly monotone, then the solution is unique; see, e.g, the existence proof in  (\cite{Showater97}, Corollary 7.1, page 84). 

To show continuity, we apply Assumption \ref{as:data}, parts \ref{c:Diff} and \ref{c:FG}, and Cauchy-Schwarz inequality and get
\[
|\left(L(B,N),(\psi,\chi)\right)|\le C \|(B,N)\|_{\mathbf{H}}\|(\psi, \chi)\|_{\mathbf{H}}, \; \mbox{for some constant $C>0$}.
\]
Recall $L$ is strictly monotone if 
\ba
\label{eq:monotone}
\left(L(B_1,N_1)-L(B_2,N_2), (B_1,N_1)-(B_2,N_2)\right)>0,\;\;\forall (B_1,N_1)\ne (B_2,N_2) \;\mbox{in}\;\; \mathbf{H}.
\ea
Also, $L$ is coercive if   
$\mbox{for some }\;\;(\psi_0,\chi_0)\in K_h\times V_h$
\ba
\label{eq:coercive}
\frac{\left(L(B,N), (B,N)-(\psi_0,\chi_0)\right)}{\|(B,N)\|_{\mathbf{H}}} \longrightarrow +\infty\;\;\mbox{as} \;\;\|(B,N)\|_{\mathbf{H}}\rightarrow \infty.
\ea

To show \eqref{eq:monotone},
we rewrite
\begin{multline}
\left(L(B_1,N_1)-L(B_2,N_2), (B_1,N_1)-(B_2,N_2) \right)
=\|B_1-B_2\|_0^2
\\
+\Delta t \left(D_B\nabla(B_1-B_2), \nabla (B_1-B_2) \right)
+\|N_1-N_2\|_0^2
\\
+\Delta t \left(D_N\nabla(N_1-N_2), \nabla (N_1-N_2)\right)%\nonumber
%\\
- \Delta t \left(F(B_1,N_1)-F(B_2,N_2), B_1-B_2\right)%\nonumber
\\
- \Delta t \left(G(B_1,N_1)-G(B_2,N_2), N_1-N_2\right).\label{eq:M1}
\end{multline}
Using Assumption \ref{as:data} part \ref{c:FG}, Cauchy-Schwarz inequality and the inequality \eqref{eq:eps}, the fifth term can be bounded as 
\begin{multline}
- \Delta t (F(B_1,N_1)-F(B_2,N_2), B_1-B_2)
\ge -\Delta t (|F(B_1,N_1)-F(B_2,N_1)|, |B_1-B_2|) %\nonumber
\\
- \Delta t (|F(B_2,N_1)-F(B_2,N_2)|, |B_1-B_2|)%\nonumber
%\\
\ge-\frac{3\Delta t}{2} M \|B_1-B_2\|_0^2 -\frac{\Delta t}{2}M \|N_1-N_2\|_0^2.\label{eq:M2}
\end{multline}
%%%
Similarly, the last term is estimated by
%%%
\begin{multline}
- \Delta t (G(B_1,N_1)-G(B_2,N_2), N_1-N_2)
\ge -\Delta t (|G(B_1,N_1)-G(B_2,N_1)|, |N_1-N_2|) %\nonumber
\\
- \Delta t (|G(B_2,N_1)-G(B_2,N_2)|, |N_1-N_2|)%\nonumber
%\\
\ge-\frac{3\Delta t}{2} M \|N_1-N_2\|_0^2 -\frac{\Delta t}{2}M \|B_1-B_2\|_0^2.\label{eq:M3}
\end{multline}
%%%
Combining \eqref{eq:M1}--\eqref{eq:M3}, and using Assumption \ref{as:data} part \ref{c:Diff} and \eqref{eq:PF}, we finally obtain
\ba
\label{eq:estimate}
\left(L(B_1,N_1)-L(B_2,N_2), (B_1-B_2, N_1-N_2)\right)
\ge
A \left(\|B_1-B_2\|_0^2+\|N_1-N_2\|_0^2\right),
\ea
with $A=(C_{PF}^2-2\Delta t MC_{PF}^2+\Delta t \gamma)$. This estimate implies \eqref{eq:monotone} 
as long as $A>0$. \add{In turn, this is guaranteed if either (i) or (ii) holds.  In fact (i) requires that the diffusivities are large enough. If this is not the case, (ii) requires a small enough $\Delta t$. }

{To show \eqref{eq:coercive}, {set $(\psi_0, \chi_0)=(0,0)$, and apply Assumption \ref{as:data} parts \ref{c:Diff}--\ref{c:FG}, and use inequality \eqref{eq:PF}. Then for all  $ (B,N)\in \mathbf{H}$, we have \begin{multline*}
    \left(L(B,N), (B,N)-(\psi_0,\chi_0)\right)=\|B\|_0^2-\Delta t(F(B,N),B)+\Delta t(D_B\nabla B, \nabla B)\\
    +\|N\|_0^2-\Delta t(G(B,N),N)+\Delta t(D_N\nabla N, \nabla N)\\
    \ge \|B\|_0^2+\|N\|_0^2 -\Delta t M(\|B\|_0^2+\|N\|_0^2)+\Delta t\gamma (\| B\|_1^2+\|N\|_1^2)\\
    =(1-\Delta t M)(\|B\|_0^2+\|N\|_0^2)+\Delta t\gamma (\| B\|_1^2+\|N\|_1^2).
    %\ge \left(Ch^2-\Delta t(MC_{PF}^2 -\gamma)\right)\|(B,N)\|_{\mathbf{H}}^2
\end{multline*}
If $\Delta t <\frac{1}{M}$, then we are done. Otherwise, apply \eqref{eq:PF}, and we have
\begin{multline*}
    \left(L(B,N), (B,N)-(\psi_0,\chi_0)\right)\ge (C_{PF}^2+\Delta t(\gamma -MC_{PF}^2))\|(B,N)\|_{\mathbf{H}}^2
    \ge A \|(B,N)\|_{\mathbf{H}}^2,
\end{multline*}
where $A$ as above. Thus, \eqref{eq:coercive} is satisfied if $\gamma > 2MC_{PF}^2$, i.e. condition (i) holds or if  $\Delta t < \frac{C_{PF}^2}{2MC_{PF}^2 -\gamma}$, which is condition (ii),
and the proof is complete.} 
 }
%% ---- -------
\end{proof}
%-------------------------------------------------------

%%-----------------------------------------------------
\subsection{Assumptions on free boundary}
The subsequent proof of convergence of the scheme \eqref{eq:DPVI} makes an assumption about the behavior of the free boundary in time. This assumption seems entirely reasonable for biofilm growth since we expect $\Omega^*$ to grow in time and its boundary to ``move forward'' rather than ``oscillate''.

Following \cite{Johnson1976} we define
\ba
\label{eq:Dn}
 D_n=\bigcup_{t\in J_n} \Omega^-(t)\cup \Omega^-(t_{n+1})\backslash \overline{\Omega^-(t)\cap\Omega^-(t_{n+1})}; \;\; n=0,\ldots N_T-1, 
 \ea 
and assume as in  (\cite{Johnson1976}, p. 601, condition 2.3) that there exists a constant $\delta$ such that
%%%
\begin{equation}
\sum_{n=0}^{N_T-1}m(D_n) \le \delta, \;\; \label{eq:g}
\end{equation}
where for a measurable $D\subset \mathbb{R}^2$, $m(D)$ denotes its Lebgesgue measure. 

We provide two examples to illustrate this assumption. 

\begin{example}
\label{ex:dn}
Consider $\Omega=(0,1)^2$, and assume that the interface $\Gamma(t) = \partial \Omega^-(t) \cap \partial \Omega^*(t)$ is a set of points $(x^*(t),y)$ with $x^*(t)$ decreasing in time starting from some $x^*(0)$, and with $y\in (0,1)$. Then  $D_n=(x^*(t_{n+1}),x^*(t_n))\times (0,1)$. As $N_T\uparrow$ we find that $\sum_n m(D_n)=m\left((x^*(t_{n+1}),x^*(0))\right)$ is increasing with $n$, but is bounded by $m(\Omega^-(0))$, thus \eqref{eq:g} holds. \end{example}
 
More generally, \eqref{eq:g} asserts that $B(x,t)$ does not change too frequently from reaching $B^*$ to lying strictly below $B^*$ or vice versa.  In particular, it avoids the following scenario. 

\begin{example}
Consider $\Omega$ as in Ex.~\ref{ex:dn}, and the scenario for some partition $\Upsilon$ of $J$ in which for every $n$, $D_n=(a,b) \times (0,1)$ with some fixed $0<a<b<0$. This scenario corresponds to the free boundary $\Gamma(t)$ oscillating between  $a$  and $b$.  Then with $N_T \uparrow$ and for finer partitions of $\Upsilon$ we find that $\sum_n m(D_n)=N_T m((a,b))$ is unbounded. 
\end{example}

%%%%%%%%%%%%%
\subsection{Error Estimate}
In this section we prove an error estimate for the approximation of solutions to \eqref{eq:PVI}. We follow the strategy of Johnson \cite{Johnson1976} which we adapt for the product space $V_h \times V_h$ and our coupled system. The main challenge is to handle the consistency error arising due to the nonlinearity of $F,G$ and the coupled nature of the system. 

We first state the main result. Next we state some auxiliary technical results, and proceed with the proof. 
Throughout, $C$ denotes a generic \add{positive} constant not depending on $\Delta t$ and $h$. We also define
\bas
e_B^n=B^n-B_h^n, e_N^n=N^n-N_h^n, n=0,\ldots, N_T. 
\eas
We first state an auxiliary result which follows from 
\eqref{eq:g}. 
%%%%
\begin{lemma}[Result from \cite{Johnson1976}, page 605; see also \cite{Azharthesis} for elaborated details]\label{lam:4}
Let $\psi\in L^\infty(W^{2,p}), \;\; 1\le p<\infty$. and assume \eqref{eq:g} holds. Then $\forall$ $\epsilon >0$, and for $m=0,\ldots, N_T$ 
\begin{equation*}
\sum_{n=0}^{m-1} \Delta t \int_{D_n} | \psi(x,t_{n+1})|\; d x  \le  \frac{\epsilon}{2}\max_n \|\psi \|_0^2 + \frac{\epsilon}{2}\sum_{n=0}^{m-1} \| \psi\|_1^2 \Delta t +C(\log \Delta t^{-1})^{1/2} \Delta t ^{3/2}.
\end{equation*}
\end{lemma}

\begin{theorem}
\label{the:convergence}{
Assume that Assumption \ref{as:data} and condition \eqref{eq:g} hold, \add{and that $\Delta t < \frac{1}{4M}$}.  Then there exists a constant $C$ independent of $\Delta t$ and $h$ such that
\begin{equation}
\max_n\left(\|e_B^n\|_0+\|e_N^n\|_0\right)+\sqrt{\sum_n\left(\|e_B^n\|_1^2+\|e_N^n\|_1^2\right)\Delta t} \le C[(\log (\Delta t)^{-1})^{1/4} \Delta t ^{3/4}+h].
\end{equation}}
\end{theorem}

\subsubsection{Auxiliary results adapted from \cite{Johnson1976}}
In the proof we will use auxiliary technical results following (\cite{Johnson1976}, page 602). In particular, we use the approximation operator $I_h$ constructed therein which applies to functions not necessarily defined pointwise. One smooths them out first, then interpolates. We only need formal properties of the operator $I_h$,  (\cite{Johnson1976}, page 602) which we restate here.  
%%%%%%%%%%%
\begin{definition} (~\cite{Johnson1976}, Lemmas 1 and 2).\label{def:1}
For each $h>0$, let $I_h:H_0^1(\Omega)\rightarrow V_h$ denote a linear operator with the following properties:
\begin{enumerate}[(i)]
  \item $\|\psi-I_h \psi\|_j \le Ch^{k-j}\|\psi\|_k,\; j=0,1,\; k=1,2,$
    \item $I_h\psi \in K_h$ if $\psi \in K.$
\end{enumerate}
\end{definition}
This operator and approximation properties are needed, e.g., to handle terms involving $\frac{\partial B}{\partial t}$. We mention that a similar goal of approximating non-smooth functions is applied in the  context of a-posteriori error estimates; see, e.g. (\cite{ErnGuer04} page 71). 

Now we define $\eta(t)=B(\cdot,t)-I_h B(\cdot,t)$, $\xi(t) = N(\cdot,t)-I_h N(\cdot,t)$ for $ t \in J$. The operator $I_h$ commutes with the time differentiation, thus the following approximation properties can be stated, as  adapted from (\cite{Johnson1976}, page 603).
%%%
\begin{lemma} (\cite{Johnson1976}, page 603)\label{lam:2}
\begin{enumerate}[(i)]
\item $\max_n\|\eta^n\|_0 \le Ch \|B\|_{L^\infty (J;H^1(\Omega))}, $
and 
 $\max_n\|\xi^n\|_0 \le Ch \|N\|_{L^\infty (J;H^1(\Omega))}. $
%--------------------------
\item $\left \|\frac{\partial \eta}{\partial t}\right \|_{L^2(J_n;L^2(\Omega))} \le Ch \left \|\frac{\partial B}{\partial t}\right\|_{L^2(J_n;H^1(\Omega))},$
and 
$\left \|\frac{\partial \xi}{\partial t}\right \|_{L^2(J_n;L^2(\Omega))} \le Ch \left \|\frac{\partial N}{\partial t}\right\|_{L^2(J_n;H^1(\Omega))}.$
%-------------------------------
\item $\|\partial \eta^n  \|_0 \le  C (\Delta t)^{-1/2}h \left \| \frac {\partial B}{\partial t}\right \|_{L^2(J_n;H^1 (\Omega))},$ and $\|\partial \xi^n  \|_0 \le  C (\Delta t)^{-1/2}h \left \| \frac {\partial N}{\partial t}\right \|_{L^2(J_n;H^1 (\Omega))}.$
%---------------------------------
\end{enumerate}
\end{lemma}

%%%
%%%
%-------------------- Proof ---------------------------

Now we proceed to the proof of Theorem~\ref{the:convergence}. 

\begin{proof}
We begin as usual deriving the error equations.  Using Assumption \ref{as:data} part \ref{c:Diff}, we have for $n=0, \ldots, N_T-1$ for the solutions of \eqref{eq:DPVI} that
%%%%%%%%%%
\begin{multline}
(\partial e_B^n, e_B^{n+1})+(\partial e_N^n, e_N^{n+1})+\mu_1(\nabla e_B^{n+1}, \nabla e_B^{n+1})
+\nu_1(\nabla e_N^{n+1}, \nabla e_N^{n+1}) %\nonumber
\\
\le(\partial e_B^n, e_B^{n+1})+(\partial e_N^n, e_N^{n+1}) +(D_B(x)\nabla e_B^{n+1}, \nabla e_B^{n+1}) %\nonumber
%\\
+(D_N(x)\nabla e_N^{n+1}, \nabla e_N^{n+1}) %\nonumber
\\
=(\partial e_B^n, \eta^{n+1})+(\partial e_N^n, \xi^{n+1})+\mu_2(\nabla e_B^{n+1},\nabla \eta^{n+1}) +\nu_2(\nabla e_N^{n+1},\nabla \xi^{n+1})%\nonumber
\\
+ (\partial B^n,I_hB^{n+1}-B_h^{n+1}) -(\partial B_h^n, I_hB^{n+1}-B_h^{n+1}) %\nonumber
\\
+ (\partial N^n,I_hN^{n+1}-N_h^{n+1}) -(\partial N_h^n, I_hN^{n+1}-N_h^{n+1}) %\nonumber
\\
+(D_B(x)\nabla B^{n+1},\nabla( I_hB^{n+1}-B_h^{n+1}))-   (D_B(x)\nabla B_h^{n+1},\nabla( I_hB^{n+1}-B_h^{n+1})) %\nonumber
\\
+(D_N(x)\nabla N^{n+1},\nabla( I_hN^{n+1}-N_h^{n+1}))-   (D_N(x)\nabla N_h^{n+1},\nabla( I_hN^{n+1}-N_h^{n+1})) .\label{eq:e}
\end{multline}

Taking $ t=t_{n+1}$, $\psi =B_h^{n+1}$, $\chi= N_h^{n+1} -N^{n+1}$ in \eqref{eq:PVI}, and $\psi = I_h B^{n+1}$, $\chi= I_h N^{n+1} - N_h^{n+1}$ in \eqref{eq:DPVI}, and adding the obtained inequalities and equations to \eqref{eq:e}, we have 
\begin{equation}
(\partial e_B^n, e_B^{n+1})+(\partial e_N^n, e_N^{n+1})+\mu_1\| e_B^{n+1}\|_1^2 +\nu_1\| e_N^{n+1}\|_1^2  \le \sum_{j=1}^{10} P_j^n, \label{eq:h}
\end{equation}
where
\begin{eqnarray}
P_1^n&=&(\partial e_B^n, \eta^{n+1}),\\
P_2^n&=&(\partial e_N^n, \xi^{n+1}),\\
P_3^n&=&\mu_2 (\nabla e_B^{n+1},\nabla \eta^{n+1}),\\
P_4^n&=&\nu_2 (\nabla e_N^{n+1},\nabla \xi^{n+1}),\\
P_5^n&=&- (D_B(x) \nabla B^{n+1},\nabla \eta^{n+1})-(\partial B^n,\eta^{n+1})+(f^{n+1},\eta^{n+1}),\\
P_6^n&=& - (D_N(x) \nabla N^{n+1},\nabla \xi ^{n+1})-(\partial N^n,\xi^{n+1})+(g^{n+1},\xi^{n+1}),\\
P_7^n&=&(\partial B^n-\frac{\partial  B}{\partial t} (t_{n+1}), e_B^{n+1}),\\
P_8^n&=&(\partial N^n-\frac{\partial  N}{\partial t} (t_{n+1}), e_N^{n+1}),\\
P_9^n&=& (F( B^{n+1} ,N^{n+1}),e_B^ {n+1})-(F(B_h^{n+1}, N_h^{n+1}), I_h B^{n+1}- B_h^{n+1}),\\
P_{10}^n&=& (G( B^{n+1} ,N^{n+1}),e_N^ {n+1})-(G(B_h^{n+1}, N_h^{n+1}), I_h N^{n+1}- N_h^{n+1}).
\end{eqnarray}
   Multiplying \eqref{eq:h} by $\Delta t$ and summing over $n=0,\ldots, m-1$;\;  $m=1,\ldots,N_T$, we obtain
\begin{multline}
\sum_{n=0}^{m-1} (e_B^{n+1}-e_B^n, e_B^{n+1})+\sum_{n=0}^{m-1}(e_N^{n+1}-e_N^n, e_N^{n+1}) 
\\
+ \mu_1 \sum_{n=0}^{m-1}\| e_B^{n+1}\|_1^2\Delta t+\nu_1 \sum_{n=0}^{m-1}\| e_N^{n+1}\|_1^2\Delta t %\nonumber 
 \le \sum_{j=1}^{10}S_j; \label {eq:i}
\end{multline}
Here we have defined $S_j=\sum_{n=0}^{m-1} |P_j^n| \Delta t$, for $j=1,\ldots, 10$. 

We also have
\begin{eqnarray}
2\sum_{n=0}^{m-1} (e_B^{n+1}-e_B^n, e_B^{n+1})&=&
\sum_{n=0}^{m-1} \|e_B^{n+1}- e_B^n\|_0^2+\|e_B^{m}\|_0^2 -\|e_B^0\|_0^2, \label{eq:k}
\end{eqnarray}
and 
\begin{eqnarray}
2\sum_{n=0}^{m-1} (e_N^{n+1}-e_N^n, e_N^{n+1})&=&\sum_{n=0}^{m-1} \|e_N^{n+1}- e_N^n\|_0^2+\|e_N^{m}\|_0^2. \label{eq:kn}
\end{eqnarray}
Multiplying \eqref{eq:i} by 2 and using \eqref{eq:k}--\eqref{eq:kn} give
\begin{multline}
   \sum_{n=0}^{m-1} \|e_B^{n+1}- e_B^n\|_0^2+\|e_B^{m}\|_0^2+ \sum_{n=0}^{m-1} \|e_N^{n+1}- e_N^n\|_0^2+\|e_N^{m}\|_0^2  \\
+2 \mu_1 \sum_{n=0}^{m-1} \| e_B^{n+1}\|_1^2 \Delta t +2 \nu_1 \sum_{n=0}^{m-1} \| e_N^{n+1}\|_1^2 \Delta t  
\le  \|e_B^0\|_0^2 +2\sum_{j=1}^{10} S_j.\label{eq:m} 
\end{multline}
We shall estimate each one of the $S_j$'s.  Many of these estimates are direct analogues of the estimates in ~\cite{Johnson1976}. Other estimates handle  the consistency and coupling terms.  

%-------------------------------------------------------

\textit{Estimation of $S_1$, and $S_2$,  analogously as in \cite{Johnson1976}}. Applying summation by parts,  Lemma \ref{lam:2}, and inequalities \eqref{eq:PF},  \eqref{eq:eps}, we obtain  
\begin{eqnarray*}
S_1&\le &  \frac{\epsilon C_{PF}^2}{2} \sum_{n=0}^{m-1} \| e_B^{n+1}\|_1^2 \Delta t+  Ch^2 \left \| \frac {\partial B}{\partial t}\right \|_{L^2(J;H^1 (\Omega))}^2+ \frac{\epsilon}{2}\|e_B^m\|_0^2 +  \epsilon\|e_B^0\|_0^2 + Ch^2 \|B\|_{L^\infty (J;H^1(\Omega))}^2. \nonumber
\end{eqnarray*}
By choosing an appropriate $\epsilon$, we have
%%%%%-------------------------------%%%%%%%
 \begin{equation}
S_1\le \frac{\mu_1}{14} \sum_{n=0}^{m-1} \| e_B^{n+1}\|_1^2 \Delta t+   \frac{1}{8}\|e_B^m\|_0^2 
+ \|e_B^0\|_0^2 + Ch^2 . \label{eq:s1}
\end{equation}
Similarly, 
 \begin{equation}
S_2 \le \frac{\nu_1}{10} \sum_{n=0}^{m-1} \| e_N^{n+1}\|_1^2 \Delta t+   \frac{1}{4}\|e_N^m\|_0^2 
+ Ch^2 . \label{eq:s2}
\end{equation}
%%------------------------------------------------------
\textit{Estimation of $S_3$, and $S_4$, analogously as in ~\cite{Johnson1976}}.
\[
S_3 \le \frac{\mu_2 \epsilon}{2}\sum_{n=0}^{m-1}  \| e_B^{n+1} \|_1^2\Delta t +  Ch^2 \|B \|_{L^\infty (J;H^2(\Omega))}^2, 
\]
where we use Lemma \ref{lam:2}. 
In particular, we have
\begin{equation}
S_3 \le \frac{\mu_1 }{14}\sum_{n=0}^{m-1}  \| e_B^{n+1} \|_1^2\Delta t +  Ch^2.  \label{eq:s3}
\end{equation}
Similarly,
\begin{equation}
S_4 \le \frac{\nu_1 }{10}\sum_{n=0}^{m-1}  \| e_N^{n+1} \|_1^2\Delta t +  Ch^2.  \label{eq:s4}
\end{equation}
%%------------------------------------------------------
%%------------------------------------------------------
\textit{Estimation of $S_5$ and $S_6$.} {These estimations are different from  those in \cite{Johnson1976}, since they apply to nonconstant diffusivities $D_B$ and $D_N$. We use Green's theorem to rewrite the first terms of $P_5^n$ and $P_6^n$ as $(\nabla D_B(x) \cdot \nabla B^{n+1}, \eta^{n+1})+(D_B(x)\Delta B^{n+1}, \eta^{n+1})$ and $(\nabla D_N(x) \cdot \nabla N^{n+1}, \xi^{n+1})+(D_N(x)\Delta N^{n+1}, \xi^{n+1})$ respectively. Then we apply Assumption \ref{as:data} parts \ref{c:Diff} and \ref{c:fg}, and we use Lemma \ref{lam:2} to obtain}

\begin{eqnarray}
S_5
\le  \sum_{n=0}^{m-1}\left(\|B^{n+1} \|_1 + \|B^{n+1} \|_2+ \|\partial B^n\|_0 +\|f^{n+1}\|_0\right)  \| \eta^{n+1}\|_0\Delta t\le C h^2. \label{eq:s5}
\end{eqnarray}

Similarly,
\begin{eqnarray}
S_6&\le& \sum_{n=0}^{m-1}\left(\|N^{n+1} \|_1 + \|N^{n+1} \|_2+ \|\partial N^n\|_0 +\|g^{n+1}\|_0\right)  \| \xi^{n+1}\|_0\Delta t\le C h^2. \label{eq:s6}
\end{eqnarray}
%-------------------------------------------------------
%-------------------------------------------------------
\textit{Estimation of $S_7$.} In estimating $S_7$, we follow Johnson's technique \cite{Johnson1976}. However,
we need to deal with coupling term. 
 We  can write $P_7^n$  as
%%%
\begin{multline*}
P_7^n= (\partial B^n-\frac{\partial  B}{\partial t} (t_{n+1}), e_B^{n+1})\\
= \frac{1}{\Delta t}\int_\Omega (B^{n+1}-B^n -\Delta t \frac{\partial B}{\partial t}(t_{n+1}) ) e_B^{n+1} \; d x\\
= \frac{1}{\Delta t}\int_\Omega \left (\int_{t_n}^{t_{n+1}}  \left( \frac{\partial B}{\partial t}(s) - \frac{\partial B}{\partial t}(t_{n+1})\right) \; d s \right ) e_B^{n+1} \; d x\\
= \frac{1}{\Delta t}\int_{t_n}^{t_{n+1}}  \left( \int_{\Omega^-}  \left(\nabla \cdot (D_B\nabla B(s)) - \nabla \cdot (D_B\nabla  B(t_{n+1}))\right) e_B^{n+1}\; d x \right)  \;d s  \\
+ \frac{1}{\Delta t}\int_{t_n}^{t_{n+1}} \left(   \int_{\Omega}  \left(\widetilde F(B(s), N(s))-\widetilde F( B^{n+1}, N^{n+1})\right) e_B^{n+1}\; d x \right)  \; d s
= q_n^1+q_n^2,
\end{multline*}
where 
\[
\widetilde F(B(t),N(t) )=\left\{\begin{array}{ll}
F(B(t),N(t))+f(x,t)  & \mbox{if $x \in \Omega^-(t)$},\\
\min(F(B, N)+f(x,t),0) & \mbox{ if $x \in \Omega^{\ast}(t)$,}
\end{array}
\right.
\]
with obvious notation for $q_n^1$ and $q_n^2$. Since $\nabla \cdot (D_B\nabla B)=0$ a.e on $\Omega^{\ast}$, by using Green's theorem we obtain
\begin{multline*}
q_n^1=\frac{1}{\Delta t}\int_{t_n}^{t_{n+1}}  \left (\int_{\Omega}  \nabla \cdot (D_B\nabla \left(B(s) -   B(t_{n+1})\right) e_B^{n+1}\; d x   \right )  \;d s\\
= \frac{1}{\Delta t}\int_{t_n}^{t_{n+1}} \int_s^{t_{n+1}} \left( \int_{\Omega}  D_B \nabla \frac{\partial B}{\partial t }(t) \cdot \nabla e_B^{n+1}\; d x   \right )  \;dt \; d s.
\end{multline*}
From the Cauchy-Schwarz inequality, we obtain
\begin{multline*}
|q_n^1| \le\frac{\mu_2}{\Delta t}\int_{t_n}^{t_{n+1}} \int_s^{t_{n+1}} \left\|   \frac{\partial B}{\partial t}(t)\right\|_1 \|  e_B^{n+1}\|_1  \;dt \; d s\\
\le\mu_2 (\Delta t)^{1/2} \|  e_B^{n+1}\|_1  \left\|  \frac{\partial B}{\partial t}\right\|_{L^2(J_n;H^1(\Omega))}  
%\\
\le\frac{\epsilon}{2}\|  e_B^{n+1}\|_1 ^2+\Delta t \frac{\mu_2^2}{2\epsilon} \left\| \frac{\partial B}{\partial t}\right\|_{L^2(J_n;H^1(\Omega))}^2.
\end{multline*}
Thus, 
%\begin{equation*}
$\sum_{n=0}^{m-1} |q_n^1|\Delta t \le \frac{\epsilon}{2} \sum_{n=0}^{m-1} \|  e_B^{n+1}\|_1 ^2 \Delta t+C (\Delta t)^2 \| \frac{\partial B}{\partial t}\|_{L^2(J;H^1(\Omega))}^2.
%\end{equation*}
$
In particular, 
%\begin{equation*}
$\sum_{n=0}^{m-1} |q_n^1|\Delta t \le \frac{\mu_1}{14} \sum_{n=0}^{m-1} \|  e_B^{n+1}\|_1 ^2 \Delta t+C (\Delta t)^2.\label{eq:q1}
%\end{equation*}
$
Now, to estimate $q_n^2$, we first notice that if $x\in \Omega \backslash D_n$, then either $x\in \Omega^-(t) \cap \Omega^-(t_{n+1})$
 or $x\in \Omega^{\ast}(t) \cap \Omega^{\ast}(t_{n+1})$ for all $t\in J_n$, so we have 
 \begin{multline*}
|\widetilde F(B(t),N(t)  )-\widetilde F(B^{n+1}, N^{n+1})| \le |F(B(t),N(t) )-  F(B^{n+1}, N^{n+1})|\\
+|f(x,t)-f(x,t^{n+1})|
\; \;\mbox{for all}\; x\in \Omega \backslash D_n.
 \end{multline*}
Thus, 
\begin{multline*}
|q_n^2|\le    \frac{1}{\Delta t}\int_{t_n}^{t_{n+1}}   \int_{\Omega \backslash D_n}  |  F( B(s),N(s) )- F( B^{n+1}, N^{n+1})|| e_B^{n+1}|\; d x   \;\; d s\\
+\frac{1}{\Delta t}\int_{t_n}^{t_{n+1}}   \int_{\Omega \backslash D_n}  |  f( x,s )- f(x, s)|| e_B^{n+1}|\; d x   \;\; d s\\
+ \frac{1}{\Delta t}\int_{t_n}^{t_{n+1}}   \int_{D_n}  (\|F\|_{L^\infty(\mathbb{R}^2)}+\|f\|_{L^\infty(J_n; L^\infty(D_n))})| e_B^{n+1}|\; d x   \;\; d s= k_n^1+k_n^2+k_n^3.
\end{multline*}
Next we estimate the terms $k_n^1+k_n^2+k_n^3$ separately.

The first term $k_n^1$ contains the coupling and nonlinearity in $F$ which are not present in \cite{Johnson1976}, To handle that, we use Assumption \ref{as:data} part \ref{c:FG}, and we have
%%%%%%%%%%%%%%%%%%%%%%%%%%%%%%%%%%%%%%%%%%%%%%%%%%%%%%
\begin{multline}
k_n^1 
\le M \frac{1}{\Delta t}\int_{t_n}^{t_{n+1}}  \int_{\Omega \backslash D_n}  |B(s) -  B(t_{n+1})| |e_B^{n+1}|\; d x  \;\; d s\\
+ M \frac{1}{\Delta t}\int_{t_n}^{t_{n+1}}  \int_{\Omega \backslash D_n}  |N(s) -  N(t_{n+1})| |e_B^{n+1}|\; d x  \;\; d s
\\
\le M \frac{1}{\Delta t}\int_{t_n}^{t_{n+1}}  \int_{\Omega \backslash D_n} \int_s^{t_{n+1} }\left| \frac{\partial B}{\partial t}(t) \right||  e_B^{n+1}| \; dt \; d x  \; d s
%\\
+ M \frac{1}{\Delta t}\int_{t_n}^{t_{n+1}}  \int_{\Omega \backslash D_n} \int_s^{t_{n+1} }\left| \frac{\partial N}{\partial t}(t) \right||  e_B^{n+1}| \; dt \; d x  \; d s
\\
\le M\frac{1}{\Delta t}\int_{t_n}^{t_{n+1}}    \int_s^{t_{n+1} } \left\| \frac{\partial B}{\partial t}(t)\right\|_0 \|   e_B^{n+1} \|_0 \; d t  \; d s
%\\
+ M\frac{1}{\Delta t}\int_{t_n}^{t_{n+1}}    \int_s^{t_{n+1} } \left\| \frac{\partial N}{\partial t}(t)\right\|_0 \|   e_B^{n+1} \|_0 \; d t  \; d s\\
\le M (\Delta t)^{1/2} \|   e_B^{n+1} \|_0 \left(  \left\| \frac{\partial B}{\partial t}\right\|_{L^2(J_n;L^2(\Omega))}+  \left\| \frac{\partial N}{\partial t}\right\|_{L^2(J_n;L^2(\Omega))}\right) \\
\le  \frac{\epsilon}{2} C_{PF}^2  \| e_B^{n+1} \|_1^2 +\frac{M^2 }{ \epsilon}\Delta t \left(\left\| \frac{\partial B}{\partial t}\right\|_{L^2(J_n;L^2(\Omega))}^2 +\left\| \frac{\partial N}{\partial t}\right\|_{L^2(J_n;L^2(\Omega))}^2\right).\label{eq:k01}
\end{multline}
%%%%%%%%%%%%%%%%%%%%%%%%%%%%%%%%%%%%%%%%%%
Next
\begin{multline}
k_n^2 =\frac{1}{\Delta t}\int_{t_n}^{t_{n+1}}   \int_{\Omega \backslash D_n}  |  f( x,s )- f(x, s)|| e_B^{n+1}|\; d x   \;\; d s\\
\le \frac{1}{\Delta t}\int_{t_n}^{t_{n+1}}   \int_{\Omega \backslash D_n} \int_s^{t_{n+1}} \left| \frac{\partial f}{\partial t} (x,t)\right|| e_B^{n+1}|\;d t\; d x   \;\; d s
%\\
\le \frac{1}{\Delta t}\| e_B^{n+1}\|_0\int_{t_n}^{t_{n+1}}  \int_s^{t_{n+1}} \left\| \frac{\partial f}{\partial t}(x,t) \right\|_0\;d t\; \; d s
\\
\le (\Delta t)^{1/2}\| e_B^{n+1}\|_0 \left\| \frac{\partial f}{\partial t}\right\|_{L^2(J_n;L^2(\Omega))}
%\\
\le \frac{\epsilon }{2}C_{PF}^2\| e_B^{n+1}\|_1^2+ \frac{1}{2\epsilon}\Delta t\left\| \frac{\partial f}{\partial t}\right\|_{L^2(J_n;L^2(\Omega))}^2.\label{eq:k02}
\end{multline}
%%%%%%%%%%%%%%%%%%%%%%%%%%%%%%%%%%%%%%%%%%%%%%%
By \eqref{eq:k01} and \eqref{eq:k02}, we have
\begin{equation}
\sum_{n=0}^{m-1} |k_n^1| \Delta t+\sum_{n=0}^{m-1} |k_n^2| \Delta t \le  \frac{\mu_1}{14} \sum_{n=0}^{m-1}  \|e_B^{n+1} \|_1^2 \Delta t+ C(\Delta t)^2 . \label{eq:k1}
\end{equation}
Finally
\begin{eqnarray*}
k_n^3&=&(\|F\|_{L^\infty(\mathbb{R}^2)}+\|f\|_{L^\infty(J_n;L^\infty(D_n))}) \int_{D_n} | e_B^{n+1}|\; d x= (\|F\|_{L^\infty(\mathbb{R}^2)}+\|f\|_{L^\infty(J_n;L^\infty(D_n))})  r_n,
\end{eqnarray*}
%%%

By Lemma \ref{lam:4}, $r_n$ can be estimated as
\[
\sum_{n=0}^{m-1}  r_n \Delta t \le  \frac{\epsilon}{2}\max_n \|e_B^{n+1} \|_0^2 + \frac{\epsilon}{2}\sum_{n=0}^{m-1} \| e_B^{n+1}\|_1^2 \Delta t +C(\log \Delta t^{-1})^{1/2} \Delta t ^{3/2}.
\]
Since at the end we want to kick the term $\max_n\|e_B^{n+1}\|_0^2$ {to the left hand side} of inequality \eqref{eq:m},  we estimate as follows.

Let  $\|e_B^{l+1}\|_0=\max_n\|e_B^{n+1}\|_0$ for some  $n,l\in \{0,\ldots, m-1\}$. Then 
\begin{multline*}
\|e_B^{l+1} \|_0^2=\sum_{j=l+1}^{m-1}\|e_B^{j+1}-e_B^j\|_0^2+\|e_B^m\|_0^2
%\\
\le\sum_{n=0}^{m-1}\|e_B^{n+1}-e_B^n\|_0^2+\|e_B^m\|_0^2.
\end{multline*}

Hence, we get
\begin{eqnarray*}
\sum_{n=0}^{m-1} r_n \Delta t &\le & \frac{\epsilon}{2}\left( \sum_{n=0}^{m-1}\|e_B^{n+1}-e_B^n\|_0^2+\|e_B^m\|_0^2\right)  + \frac{\epsilon}{2}\sum_{n=0}^{m-1} \| e_B^{n+1}\|_1^2 \Delta t +C(\log \Delta t^{-1})^{1/2} \Delta t ^{3/2},
\end{eqnarray*}
which yields 
\begin{multline}
\sum_{n=0}^{m-1} k_n^3 \Delta t 
\le \frac{1}{4}  \sum_{n=0}^{m-1}\|e_B^{n+1}-e_B^n\|_0^2+\frac{1}{8} \|e_B^m\|_0^2
\\
+ \frac{\mu_1}{14}\sum_{n=0}^{m-1} \| e_B^{n+1}\|_1^2 \Delta t +C(\log \Delta t^{-1})^{1/2} \Delta t ^{3/2}.  \label{eq:rnn}
\end{multline}
%-------------------------------------------------------

\textit{Estimation of $S_8$.} Using Cauchy-Schwarz inequality, we have
\[
S_8 \le \frac{\epsilon}{2}\sum_{n=0}^{m-1} \|e_N^{n+1}\|_1^2 \Delta t +\frac{1}{2\epsilon} \sum_{n=0}^{m-1} \left\|\partial N^n-\frac{\partial N}{\partial t}(t_{n+1})\right\|_0^2 \Delta t.
\]
By Assumption \ref{as:data} part \ref{c:sol} on $N$, we have $\frac{\partial^2 N}{\partial t^2} \in L^2(Q)$, thus the second term of the above inequality can be estimated as
\begin{multline*}
\left\|\partial N^n-\frac{\partial N}{\partial t}(t_{n+1})\right\|_0^2 \Delta t 
\le  \int_\Omega \int_{t_n}^{t_{n+1}} \left( \frac{\partial N}{\partial t}(s)-\frac{\partial N}{\partial t}(t_{n+1}) \right)^2  \; d s\; d x \\ 
\le  \Delta t \int_\Omega \int_{t_n}^{t_{n+1}}  \int_s^{t_{n+1}} \left(\frac{\partial^2 N}{\partial t^2}(t)\right)^2  \;dt   \; ds\; dx 
%\\ 
= (\Delta t)^2 \left\|\frac{\partial^2 N}{\partial t^2}\right\|_{L^2(J_n;L^2(\Omega))},
\end{multline*}
which implies 
$
\sum_{n=0}^{m-1} \left\|\partial N^n-\frac{\partial N}{\partial t}(t_{n+1})\right\|_0^2 \Delta t \le (\Delta t)^2 \left\|\frac{\partial^2 N}{\partial t^2}\right\|_{L^2(J;L^2(\Omega))}.
$
Therefore we have 
$
S_8 \le \frac{\epsilon}{2}\sum_{n=0}^{m-1} \|e_N^{n+1}\|_1^2 \Delta t +C(\Delta t)^2. 
$  
In particular,
\begin{equation}
S_8 \le \frac{\nu_1}{10}\sum_{n=0}^{m-1} \|e_N^{n+1}\|_1^2 \Delta t +C(\Delta t)^2. \label{eq:s8}
\end{equation}

%%------------------------------------------------------
%%------------------------------------------------------
 \textit{Estimation of $S_9$,  and $S_{10}$.} These are the consistency terms. By Assumption \ref{as:data} part \ref{c:FG}, $P_9^n$ can be estimated as 
\begin{multline*}
P_9^n= (F( B^{n+1} ,N^{n+1})-F( B^{n+1} ,N_h^{n+1}),e_B^ {n+1})+ (F( B^{n+1} ,N_h^{n+1})-F( B_h^{n+1} ,N_h^{n+1}),e_B^ {n+1})\\
 -(F(B_h^{n+1}, N_h^{n+1}), \eta^{n+1})\\
\le M \|e_N^{n+1}\|_0\|e_B^{n+1}\|_0+ M  \|e_B^{n+1}\|_0^2+ \|F\|_{L^\infty(\mathbb{R}^2 )}\|\eta^{n+1}\|_0\\ 
\le \frac{M^2 \epsilon}{2}\|e_N^{n+1}\|_0^2+\frac{1}{2\epsilon}\|e_B^{n+1}\|_0^2+ M\|e_B^{n+1}\|_0^2+ \|F\|_{L^\infty(\mathbb{R}^2)}m(\Omega)^{1/2}\|\eta^{n+1}\|_0.
\end{multline*}
Hence, we obtain
\begin{equation}
S_9 \le \frac{\nu_1}{10}\sum_{n=0}^{m-1}\|e_N^{n+1}\|_1^2\Delta t+\frac{\mu_1}{14}\sum_{n=0}^{m-1}\|e_B^{n+1}\|_1^2\Delta t + M\sum_{n=0}^{m-1}\|e_B^{n+1}\|_0^2 \Delta t+ Ch^2. \label{eq:s9}
\end{equation}
%%----------------------------------------------------
Similarly,
\begin{equation}
S_{10} \le \frac{\nu_1}{10}\sum_{n=0}^{m-1}\|e_N^{n+1}\|_1^2\Delta t+\frac{\mu_1}{14}\sum_{n=0}^{m-1}\|e_B^{n+1}\|_1^2\Delta t + M \sum_{n=0}^{m-1}\|e_N^{n+1}\|_0^2 \Delta t+ Ch^2. \label{eq:s10}
\end{equation}
%%-------------------------------------------------

Now we collect all the above estimates from \eqref{eq:m}--\eqref{eq:s10} and we get
\begin{multline*}
\sum_{n=0}^{m-1} \|e_B^{n+1}- e_B^n\|_0^2+\sum_{n=0}^{m-1} \|e_N^{n+1}- e_N^n\|_0^2+\|e_B^{m}\|_0^2\\
+\|e_N^{m}\|_0^2+ \mu_1 \sum_{n=0}^{m-1} \| e_B^{n+1}\|_1^2 \Delta t + \nu_1 \sum_{n=0}^{m-1} \| e_N^{n+1}\|_1^2 \Delta t\\
\le 2 \|e_B^0\|_0^2 + C\left(h^2+(\log \Delta t^{-1})^{1/2} \Delta t ^{3/2} \right)
%\\
+  \sum_{n=0}^{m-1} 2M\left( \|e_B^{n+1} \|_0^2+ \|e_N^{n+1} \|_0^2\right) \Delta t. \end{multline*}
Thus we have {
\begin{multline*}
\|e_B^{m}\|_0^2
+\|e_N^{m}\|_0^2+ \mu_1 \sum_{j=1}^{m} \| e_B^{j}\|_1^2 \Delta t + \nu_1 \sum_{j=1}^{m} \| e_N^{j}\|_1^2 \Delta t\\
\le 2 \|e_B^0\|_0^2 + C\left(h^2+(\log \Delta t^{-1})^{1/2} \Delta t ^{3/2} \right)
%\\
+  \sum_{j=1}^{m} 2M\left( \|e_B^{j} \|_0^2+ \|e_N^{j} \|_0^2\right) \Delta t. 
\end{multline*}
By the assumption on the smallness of $\Delta t$, we have $2M \Delta t \left(\|e_B^m\|_0^2+\|e_N^m\|_0^2\right)<\frac{1}{2}\left(\|e_B^m\|_0^2+\|e_N^m\|_0^2\right)$, and hence we have
\begin{multline}
\|e_B^{m}\|_0^2
+\|e_N^{m}\|_0^2+ 2\mu_1 \sum_{j=1}^{m} \| e_B^{j}\|_1^2 \Delta t + 2\nu_1 \sum_{j=1}^{m} \| e_N^{j}\|_1^2 \Delta t\\
\le  C\left(h^2+(\log \Delta t^{-1})^{1/2} \Delta t ^{3/2} \right)
%\\
+  \sum_{j=1}^{m-1} 4M\left( \|e_B^{j} \|_0^2+ \|e_N^{j} \|_0^2\right) \Delta t. \label{eq:col2}
\end{multline}
We apply the discrete Gronwall's Lemma (\cite{BD2007}, inequality (3.27), page 1114)
\begin{multline}
    (r^0)^2+(s^0)^2\le (p^0)^2, \;\;\text{and}\;\;
    %\\
    (r^m)^2+(s^m)^2\le \sum_{j=0}^{m-1}(y^j)^2(r^j)^2+\sum_{j=0}^m(p^j)^2, \;\; m\ge 1\\
    \Longrightarrow (r^m)^2+(s^m)^2\le\exp\left(\sum_{j=0}^{m-1}(y^j)^2\right)\sum_{j=0}^m(p^j)^2, \;\; m\ge 1.\label{eq:Gronwall2}
\end{multline}}
{with $(r^m)^2=\|e_B^{m}\|_0^2
+\|e_N^{m}\|_0^2$,
$(s^m)^2=2\mu_1 \sum_{j=1}^{m} \| e_B^{j}\|_1^2 \Delta t + 2\nu_1 \sum_{j=1}^{m} \| e_N^{j}\|_1^2\Delta t$,
$(p^j)^2=(t_j-t_{j-1}) C \left(h^2+(\log \Delta t^{-1})^{1/2} \Delta t ^{3/2} \right)$, $(y^j)^2=4M\Delta t$ to \eqref{eq:col2} to get
\begin{multline*}
    \|e_B^{m}\|_0^2
+\|e_N^{m}\|_0^2+ 2\mu_1 \sum_{j=1}^{m} \| e_B^{j}\|_1^2 \Delta t + 2\nu_1 \sum_{j=1}^{m} \| e_N^{j}\|_1^2 \Delta t\\
\le  C\left(h^2+(\log \Delta t^{-1})^{1/2} \Delta t ^{3/2} \right)\exp(\sum_{j=0}^{m-1} 4 M\Delta t) 
%\\
\le C\left(h^2+(\log \Delta t^{-1})^{1/2} \Delta t ^{3/2} \right)\exp (4MT).
\end{multline*}
Therefore, there exists a constant $C$ independent of $h$ and $\Delta t$ such that
\[
     \|e_B^{m}\|_0^2
+\|e_N^{m}\|_0^2+  \sum_{j=1}^{m} \| e_B^{j}\|_1^2 \Delta t +  \sum_{j=1}^{m} \| e_N^{j}\|_1^2 \Delta t\\
\le C\left(h^2+(\log \Delta t^{-1})^{1/2} \Delta t ^{3/2} \right),
\]
and the proof is complete.
}
\end{proof}
%%-------------------------------------------------------------------
%%-------------------------------------------------------------------

\begin{remark}
The order of convergence is actually close to first order in $h$ if $h=O(\Delta t)$. 
\end{remark}

%%%%%%%%%%%%%%%%%%%%%%%%%%%%%%%%%%%%%%%%%%%%%%%%
\section{Numerical Solver} \label{sec:solver}
The problem \eqref{eq:DPVI} is a system of nonlinear equations to be solved for $(B_h^n, N_h^n)$ at each time step $n$. Below we formulate this algebraic system and discuss the nonlinear solver which has the structure of Newton-Raphson iterations complemented with the constraint implicit in the variational inequality  \eqref{eq:DPVI}. We work in the framework of nonlinear complementarity constraints (NCC), and with a Lagrange multiplier $\Lambda_h^n$.   Given some $B_*<B^*$ from $\mathbb{R}\cup \{\infty,-\infty\}$ we define the Evans function
\ba
\label{eq:evans}
P_{[B_\ast,B^\ast]}(\psi)=\max\{B_\ast,\min(\psi,B^\ast)\}.
\ea
This function is piecewise differentiable. In the model considered in this paper $B_*=-\infty$, but the solver can be applied easily to the general ``double-obstacle'' case.  

To show how the Newton solver is applied to the solution of PVI, we rewrite \eqref{eq:DPVI} in an equivalent algebraic form. 
%%%%%%%%%%%%%%%%%%%%%%%%%%%%%%%%%%%%%%%%%
Let $V_h=\mathrm{span} \{\phi_i\}_{i=1}^q$ where $q$ is the number of degrees of freedom. \add{Define the vectors of degrees of freedom $\mathbf{B}^n,\mathbf{N}^n$ of $B_h^{n},N_h^{n}$. Define the matrices $\mathbf{M},\; \mathbf{R},\; \mathbf{A}_B,\; \mathbf{A}_N$, each with entries with subscript $(ij)$}
\bas
M_{ij}&=&\int_\Omega \phi_i \phi_j, \;\;
R_{ij}=\int_\Omega P(N^n_h) \phi_i \phi_j,
\\
(A_B)_{ij}&=&\int_\Omega D_B({B^n_h}) \nabla\phi_i \cdot \nabla \phi_j, \;\;
(A_N)_{ij}=\int_\Omega D_N({B^n_h}) \nabla\phi_i \cdot \nabla \phi_j,
\eas
%%%
where we applied time-lagging. For $F$ and $G$ we use \eqref{eq:monod} and define vectors $\mathbf{F_B}$ and $\mathbf{F_N}$ with the entries 
  $(F_B)_i=\int_\Omega f(x,t_{n+1})\phi_i$, $(F_N)_i=\int_\Omega g(x,t_{n+1})\phi_i$.  We also seek the vector $\Lambda^{n+1}$ of pointwise values on the grid. 

In summary \eqref{eq:DPVI} can be written in the form
\begin{subequations}
\label{eq:system1}
\begin{align}
(\mathbf{M}+\;\Delta t \;\mathbf{A_B}-\kappa_B\Delta t\; \mathbf{R}\;)\;\mathbf{B}^{n+1}-\Delta t\; \mathbf{M} \Lambda^{n+1}-\Delta t \mathbf{F_B}-\mathbf{M} \mathbf{B}^n&=0, \label{eq:sys1}\\
\mathbf{B}-P_{[B_\ast,B^\ast]}(\mathbf{B}^{n+1}-\Lambda^{n+1})&=0;\label{eq:sys2}\\
(\mathbf{M}+\;\Delta t \;\mathbf{A_N})\mathbf{N}^{n+1}+\kappa_N\Delta t\; \mathbf{R}\mathbf{B}^{n+1}-\Delta t \mathbf{F_N}-\mathbf{M} \mathbf{N}^n&=0. \label{eq:sys3}
\end{align}
\end{subequations}

\begin{remark}
The equations \eqref{eq:sys1} and \eqref{eq:sys3} are the discrete counterparts of the diffusion--reaction equations at each time step. The middle equation \eqref{eq:sys2} enforces pointwise that $B_\ast \le B_h^{n+1}\le B^\ast$.  The Lagrange multiplier is in fact a selection
%%%%%%%
$
-\Lambda_h^{n+1}\in \partial I_{[B_\ast,B^\ast]}(B_h^{n+1})
$.
\end{remark}

%%%%%%%%%%%%%%%%%%%%%%%%%%%%%%%%%%
\add{In summary, \eqref{eq:system1} can be written together
with $\mathcal{F}:\mathcal{R}^{3q}\rightarrow\mathcal{R}^{3q}$
\[
\mathcal{F}(\mathbf{B}^{n+1}, \Lambda^{n+1}, \mathbf{N}^{n+1})=0,
\]
where $\mathcal{F}$ is  semismooth because of the mixed complementarity condition \eqref{eq:sys2} expressed with the semismooth Evans function \eqref{eq:evans}; see \cite{Ulbrich2011}. 
We solve this problem by iteration for  $(\overline{ \mathbf{B}^{n+1}},\overline{ \Lambda^{n+1}},\overline{ \mathbf{N}^{n+1}}) \approx (\mathbf{B}^{n+1},\Lambda^{n+1},\mathbf{N}^{n+1})$ only up to certain tolerance $tol$. The algorithm starts with $( \mathbf{B}^0,0,\mathbf{N}^0)$ from initial data. }
\begin{algorithm}
\label{alg:semi}
(Semismooth Newton Method for \eqref{eq:system1}, at every time step $n\geq 0$) 

Given $(\overline{ \mathbf{B}^n},\overline{\Lambda^{n}},\overline{\mathbf{N}^n}$), solve for $(\overline{\mathbf{B}^{n+1}},\overline{\Lambda^{n+1}},\overline{\mathbf{N}^{n+1}})$ as shown below. 

Denote the iterative guesses by 
$({\mathbf{B}^{(k)}},\Lambda^{(k)}, {\mathbf{N}^{(k)}})$. 
%%%%
\begin{itemize}
    \item [Step (0).] Choose an initial guess  ${ \mathbf{B}^{(0)}}=\overline{\mathbf{B}^n}$, $\Lambda^{(0)}=\overline{\Lambda^n}$, and ${\mathbf{ N}^{(0)}}=\overline{\mathbf{N}^n}$, and set $k=0$. 
    \item [Step (1).] Evaluate current residual 
    $\mathcal{F}^{(k)}= \mathcal{F}({\mathbf{B}^{(k)}}, \Lambda^{(k)}, {\mathbf{N}^{(k)}})$.    
    \\
    If $\|\mathcal{F}^{(k)}\|_\infty< tol,$ let $K=k$, and go to Step 5.
\item [Step (2).] Evaluate current Jacobian. Select some $\mathcal{J}^{(k)} \in \partial \mathcal{F}( \mathbf{B}^{(k)}, \Lambda^{(k)},  \mathbf{N}^{(k)})$. 
    \item [Step (3).] {Solve} for the correction $\mathcal{J}^{(k)} s_k =-\mathcal{F}^{(k)}$; where $s^k=(s_{B}^k,s_{\Lambda}^k, s_{N}^k)$.
    \item [Step (4).] Correct the current guess: set 
    \bas
    (\mathbf{B}^{(k+1)},\Lambda^{(k+1)}, \mathbf{N}^{(k+1)})=( \mathbf{B}^{(k)},\Lambda^{(k)},  \mathbf{N}^{(k)})+(s_{B}^{(k)},s_{\Lambda}^{(k)}, s_{N}^{(k)}).
    \eas
    Increment $k$ by one, and go to Step 1.
  \item [Step (5).] Set  $\overline{\mathbf{B}^{n+1}}= \mathbf{B}^{(K)}$, $\overline{\Lambda^{n+1}}=\Lambda^{(K)}$ and $\overline{\mathbf{N}^{n+1}}= \mathbf{N}^{(K)}$. STOP. 
\end{itemize}
\end{algorithm}

The Algorithm~\ref{alg:semi} is very efficient and performs in a way superior to other solvers such as, e.g. \textit{relaxation methods} given in \cite{Glowinski1980}. See \cite{Azharthesis} for our comparison of these solvers for a model PVI}. \add{We use $tol=10^{-6}$. Since the average number of Newton iterations is between $2$ and $3$, we skip the detailed report.} In general, the solver can be proven to converge locally $q$-superlinearly \cite{Ulbrich2011}. 
\add{In our implementation we use MATLAB, and its built-in linear solver for sparse systems.}

%----------------------------------------------

%%%%%%%%%%%%%%%%%%%%%%%%%%%%%%%%%%%%%%%%%%%%%%%%
\section{Numerical Experiments} \label{sec:NE}
In this section we present numerical experiments designed to show convergence at the rates predicted by Theorem~\ref{the:convergence} for $d$=$2$. We also show convergence in the cases not covered by the theory; in particular, we consider $d$=$1$, $d$=$3$, as well as Neumann boundary conditions. 

Further, we consider some simulations with nonlinear diffusivities as well as in irregular geometries similar to those encountered at the porescale. Our experiments are motivated by the experimental set-up in \cite{PTISW2016,CVTAKPT}. We illustrate the pointwise behavior of the biofilm and nutrient and of the total amount 
\ba
\bar B(t)=\int_{\Omega} B(x,t) \;dx
\ea
of biomass depending on modeling assumptions. 
For some simulations we present both the approximations to $B(x,t)$ as well as of $N(x,t)$. In some other simulations, the illustrations of $N(x,t)$ are predictable and are skipped. 
In captions, we abbreviate ``boundary conditions'' to ``bc''. 

\medskip
We examine the approximation error in two norms
analyzed in Theorem~\ref{the:convergence}.
The first 
ERR1$^0$=$\max_n (\|e_B^n\|_0+\|e_N^n\|_0)$, and the second error quantity %%%
ERR2$^0$=$\sqrt{\sum_n\left(\|e_B^n\|_1^2+\|e_N^n\|_1^2\right)\Delta t}$ for a scalar problem is shown in \cite{Johnson1976} to be of the same order as ERR1$^0$. 

\begin{remark}
In practice, we are unable to verify the convergence exactly using ERR1$^0$ and ERR2$^0$, because the true solutions $B(x,t)$ and $N(x,t)$ to our coupled system are not known.  \add{Manufacturing solutions under constraints is nontrivial, and doing so for the coupled system is even more complicated except in scalar examples in \cite{Azharthesis} or physically unrealistic examples}. Thus we choose fine grid solutions $B_{fine},N_{fine}$ as surrogates for $B,N$, with some $h_{fine}$ significantly smaller than $h$ considered in the convergence study. Unfortunately, this approach requires also an appropriately small $\Delta t_{fine}$, resulting in a very large number of time steps. 
\\
Therefore, our ability to actually check the convergence over all time steps $n=1,2,\ldots$ as indicated in ERR1$^0$ and ERR2$^0$ with these large numbers of time steps is limited. Instead, we limit ourselves to the sampling of the spatial errors in time only over a limited set of a selected few $K$ time steps $\Upsilon=\{T_1,T_2,\ldots T_K\}$ which correspond to some selected indices $\{N_1,N_2,\ldots N_K\}$, different for each $\Delta t$. In what follows we report
\bas
ERR1^\Upsilon&=&\max_{n \in \{N_1,N_2,\ldots N_K\} } (\|e_B^n\|_0+\|e_N^n\|_0),
\\
ERR2^\Upsilon&=&\sqrt{\sum_{n \in \{N_1,N_2,\ldots N_K\} }
\left(\|e_B^n\|_1^2+\|e_N^n\|_1^2\right)\Delta t},
\eas
and in each instance we indicate which $T_1,T_2,\ldots T_K$ are in $\Upsilon$.
\end{remark}

\begin{remark}
\label{rem:sparse}
It is well known that using fine grid solution may somewhat overpredict the convergence rate. In addition, in our examples the  sampling of the error in time is quite sparse, therefore we expect to see convergence rate higher than that predicted by the theorem. 
\end{remark}

%%%%%%%%%%%%%%%%%%%%%%%%%%%%%%%%%%%%%%%%
\subsection{1D experiments with Dirichlet Boundary Conditions}
\label{sec:1d}

We start with a simple model problem in 1d with homogeneous Dirichlet boundary conditions. The data in this model problem satisfies exactly the conditions in  Assumption~\ref{as:data}, but strictly speaking the case $d$=$1$ is not covered by Theorem~\ref{the:convergence}.

\begin{example}[1D Simulation] \label{ex:1D}
 Let $\Omega=(0,1)$, and let the diffusion coefficients be constant $D_B$=$0.5$, $D_N$=$0.1$. 
 We set the initial biofilm $B_{init}(x)=0.01 |sin(\pi x)|$,  the initial nutrient   $N_{init}(x)=0.02\chi_{(0.25, 0.75)}$. 
 We use Monod functions $F(B,N)=\frac{2500N}{N+0.7}B$ and  $G(B,N)=-\frac{100N}{N+0.7}B$. We also set $B^\ast=0.02$, and use boundary condition $B(0,t)=B(1,t)=0=N(0,t)=N(1,t)$. \end{example}

    %---------------------------------------------
 %%%%%%%%%%%%%%%%%%%%%%%%%%%%%%%%%%%%%%
 \begin{figure}
    \begin{center}
    \includegraphics[scale=0.34]{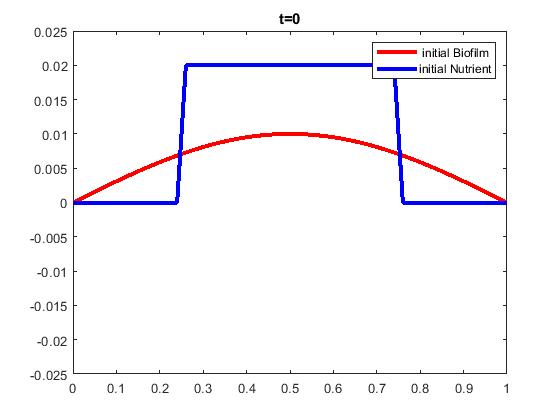}
   \includegraphics[scale=0.34]{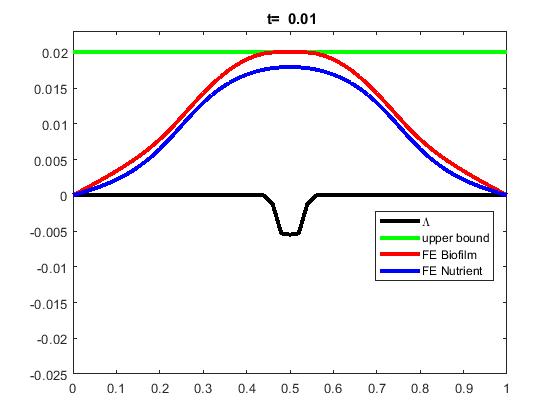}
    \includegraphics[scale=0.34]{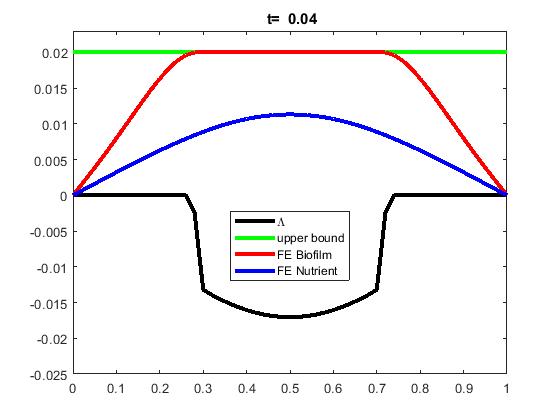}
    \includegraphics[scale=0.34]{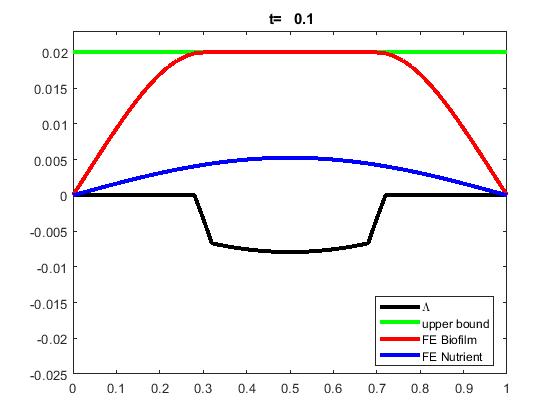}
    \caption{Evolution in Ex.~\ref{ex:1D} with $h=0.02$, and $\Delta t=0.005$. } 
    \label{fig:1D}
    \end{center}
    \end{figure}
    %

%%%%%%%%%%
\subsubsection{}
Fig.~\ref{fig:1D} shows the typical growth of biofilm  and the decay of  nutrient over time up until about $T \approx 0.02$. Since the nutrient is initially concentrated in the middle of the domain, the majority of growth of $B$ and decay of $N$ occurs there. Eventually the nutrient diffuses away, and the biomass starts also growing elsewhere. 
 
At $t=0.02$ the biomass reaches the maximum density $B^*$, and the biofilm ``phase'' in $\Omega^*$ forms.
The Lagrange multiplier $\Lambda$ is active when needed in $\Omega^*(t)$ to enforce the constraint. \add{The Lagrange multiplier takes always negative values, since its role (when moved to the right hand-side) is to ``push the solution down'' so that \eqref{eq:constraint} is satisfied.}
The biofilm starts growing  through the interface moving as a free boundary; this scenario is similar to that described in Ex.~\ref{ex:dn}. This shows that \eqref{eq:g} is a reasonable assumption. 

The nutrient is not consumed in $\Omega^*$, but it slowly diffuses away \add{towards the external boundaries at $x=0$ and $x=1$ through which it escapes, and  towards the region $\Omega^-$ where it is consumed by the growing biomass.
}

\subsubsection{} 
To test convergence of the numerical model, we use
a fine grid solution with $h=0.001$ and $\Delta t=0.0001$ as a surrogate for an analytical solution. Since it would be difficult to set up $\Delta t$ to make the logarithmic terms $
[(\log (\Delta t)^{-1})^{1/4} \Delta t ^{3/4}]$ conform to $h$, we choose 
$\Delta t=O(h)$.  
We present the errors with $\Upsilon=\{0.05,0.1\}$. Table \ref{tab:1D} shows $ERR1^\Upsilon$  and $ERR2^\Upsilon$. 

The results demonstrate an essentially first order of convergence in ERR1$^\Upsilon$ in $h$, a bit higher than that predicted by the Theorem~\ref{the:convergence} for $d$=$2$, likely thanks to an additional regularity of the solution exhibited by a modest size of the region $D_n$ discussed in \eqref{eq:Dn}, and due to our error sampling strategy discussed in Remark~\ref{rem:sparse}. 
On the other hand, we see that the convergence order in ERR2 is  about $O(h^{3/2})$. These results can be compared to those theoretically predicted in \cite{Johnson1976} as well as our numerical results for that theoretical case shown in the Appendix. 

%%%%%%%%%%%%%%%%%%
\begin{table}
\begin{center}
 \begin{tabular}{c c c c c c }	
\hline
 $h$  & $\Delta t$ & $ERR 1$ &$ERR 2$ & $ERR 1$ order &   $ERR 2$ order\\
\hline \hline
0.01  &       0.01 &     0.00026 &   0.00028  &          &         \\
0.005&    0.005 &   0.00014 &   0.00010 &   0.95433  &   1.4365 \\
 0.0025&     0.0025 &   6.6251e-05  &  3.6292e-05 &    1.0347 &    1.4961\\ 
    \hline
  \end{tabular}
  \caption{Convergence for Ex.~\ref{ex:1D}; Dirichlet bc}
  \label{tab:1D}
\end{center}
\end{table}
%---------------------------------------------------

%%%%%%%%%%%%%%%%%%%%%%%%%%%%%%%%%%

\subsection{Experiments in $d=1$, with homogeneous Neumann bc and various choices of diffusivity}

Here we focus on the choice of diffusivity $D_B$, since this choice is the primary difference between the \add{various models in the literature}. In particular, we set up experiments with Neumann conditions to model the growth in an isolated system, the only difference among the experiments being the choice of $D_B(B)$. 

{
\begin{example}[1D Simulation, growth with Neumann conditions]
\label{ex:1DNeu}
  We fix the domain $\Omega=(0,1)$, $N_{init}=\chi_{\Omega}$, $B_{init}=0.2\chi_{(0.4,0.6)}$, the diffusivity of nutrient $D_N=0.5$, the maximum density $B^*=0.03$, the growth and utilization  functions $F(B,N)=\frac{10 N}{N+0.007}B$, $G(B,N)=\frac{5 N}{N+0.007}B$, respectively. We vary the diffusivity of biofilm $D_B$. We consider four cases (i) $D_B=0.1$, (ii) $D_B=0.001$,  (iii) $D_B=(D_{\max}-D_{\min})(B/B^*)+D_{\min}$, (iv) $D_B=(D_{\max}-D_{\min})(B^8/B^*)+D_{\min}$;  $D_{\max}=0.1$, $D_{\min}=0.001$. 
\end{example}}

%%%%%%%%
\begin{figure}
    \centering
    \includegraphics[scale=0.3]{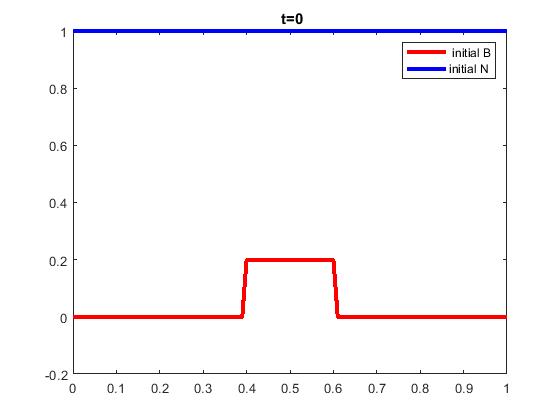}
    \includegraphics[scale=0.3]{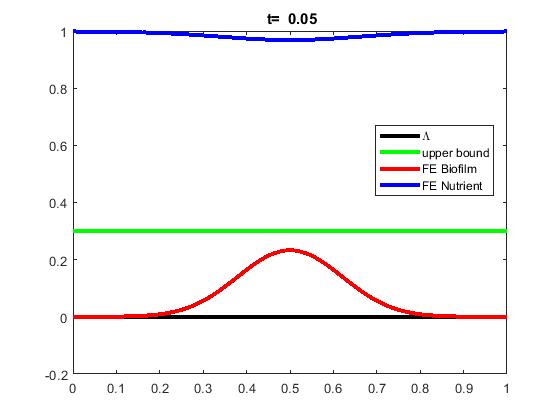}
    \includegraphics[scale=0.3]{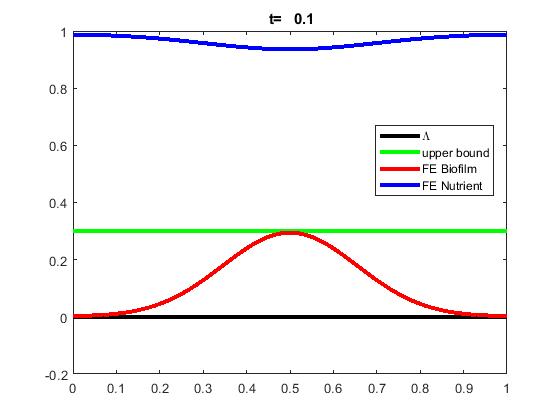}
    \includegraphics[scale=0.3]{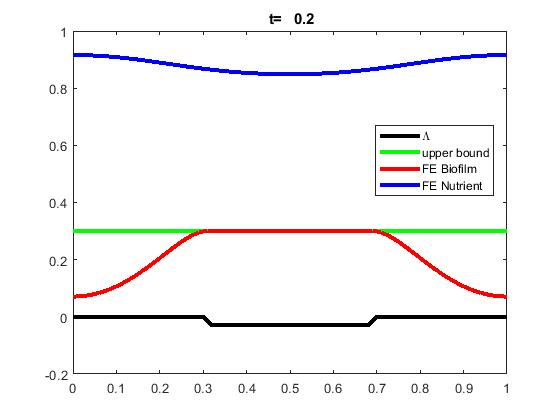}
    \caption{Evolution in Ex.~\ref{ex:1DNeu} with $D_B=0.1$, $h=0.01$, $\Delta t=0.002$}
    \label{fig:1DNeu_i}
\end{figure}
%%%%%%%%%%%%%%%%%%%%%%%%%%%%%%%%%%%%%%%%%%
\begin{figure}
    \centering
      \includegraphics[scale=0.3]{pic/fig5t0.jpg}
      \includegraphics[scale=0.3]{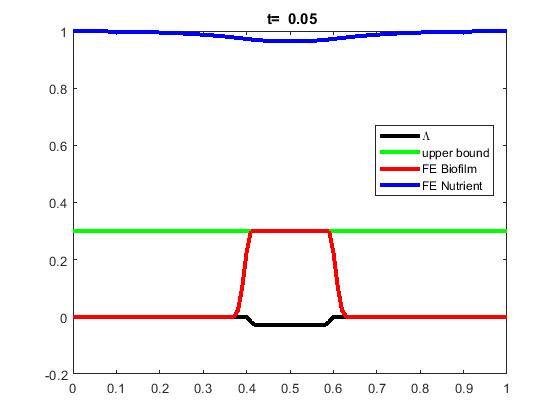}
      \includegraphics[scale=0.3]{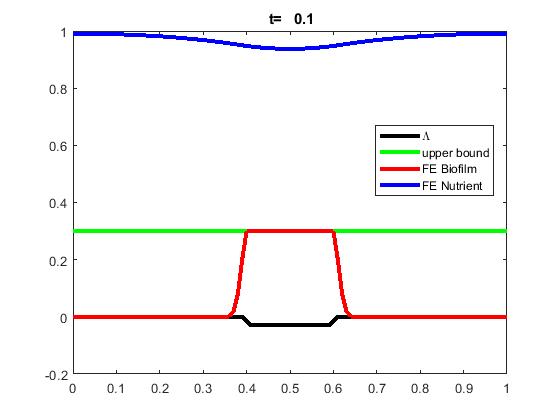}
      \includegraphics[scale=0.3]{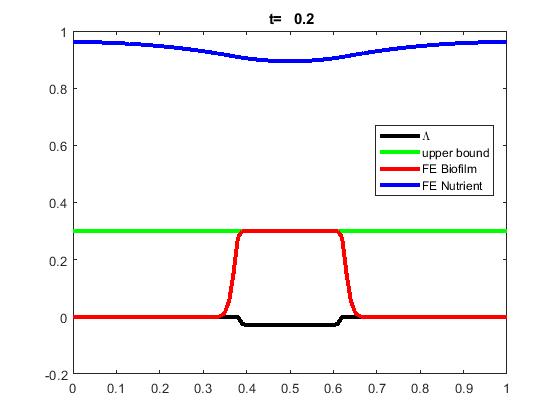}
    \caption{Evolution in Ex.~\ref{ex:1DNeu} with $D_B=0.001$, $h=0.01$, $\Delta t=0.0002$}
    \label{fig:1DNeu_ii}
\end{figure}
%%%%%%%%%%%%%%%%%%%%%%%%%%%%%%%%%%%%%%%%%%
\begin{figure}
    \centering
    \includegraphics[scale=0.3]{pic/fig5t0.jpg}
    \includegraphics[scale=0.3]{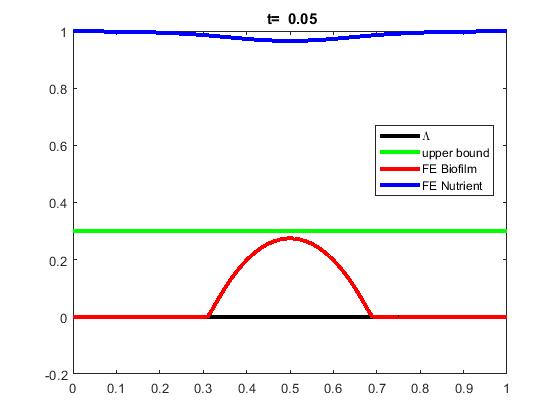}
    \includegraphics[scale=0.3]{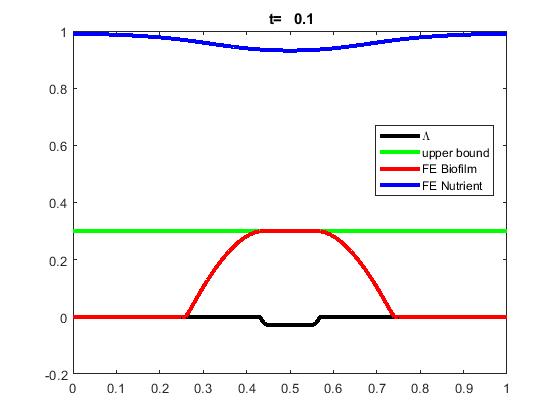}
    \includegraphics[scale=0.3]{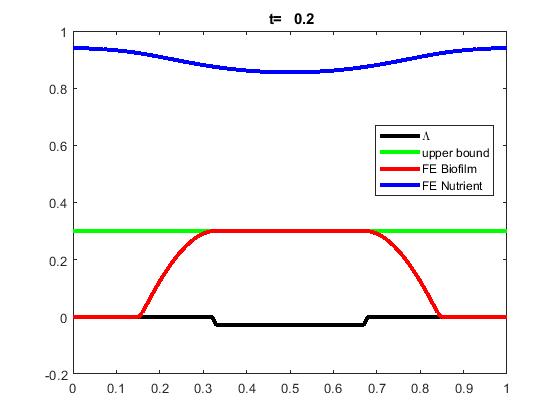}
    \caption{Evolution in Ex.~\ref{ex:1DNeu} with $D_B=(D_{\max}-D_{\min})(B/B^*)+D_{\min}$;  $D_{\max}=0.1$, $D_{\min}=0.001$, $h=0.01$, $\Delta t=0.0002$}
    \label{fig:1DNeu_iii}
\end{figure}
%%%%%%%%%%%%%%%%%%%%%%%%%%%%%%%%%%%%%%%%%%%
\begin{figure}
    \centering
    \includegraphics[scale=0.3]{pic/fig5t0.jpg}
    \includegraphics[scale=0.3]{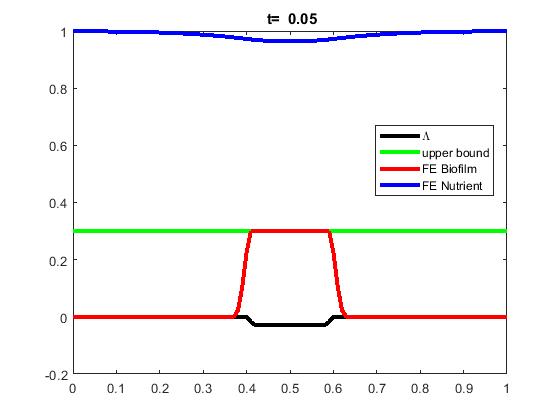}
    \includegraphics[scale=0.3]{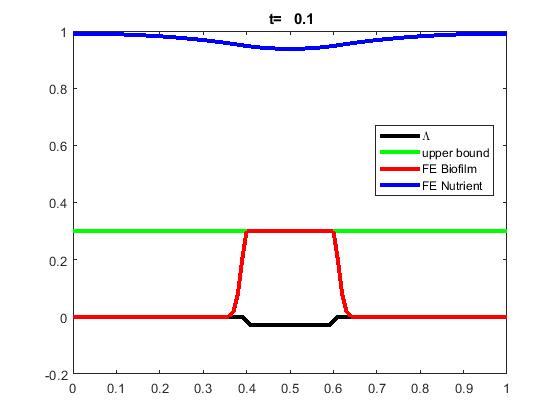}
    \includegraphics[scale=0.3]{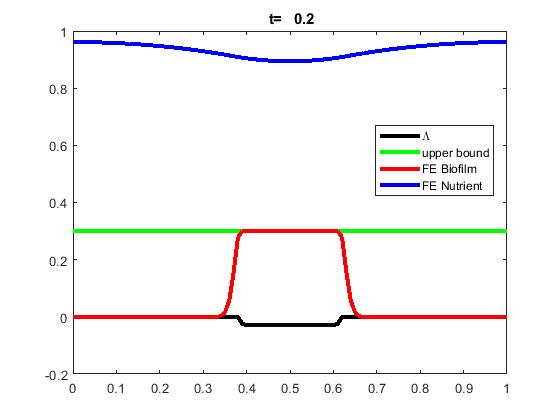}
    \caption{Evolution of $B(x,t)$, $N(x,t)$ in Ex.~\ref{ex:1DNeu} with $D_B=(D_{\max}-D_{\min})(B^8/B^*)+D_{\min}$;  $D_{\max}=0.1$, $D_{\min}=0.001$, $h=0.01$, $\Delta t=0.0002$}
    \label{fig:1DNeu_iv}
\end{figure}
%%%%%%%%%%%%%%%%%%%%%%%%%%%%%%%%%%%%%%%%

See the evolution of biofilm and nutrient in Fig.~\ref{fig:1DNeu_i}--\ref{fig:1DNeu_iv}. The difference in evolution patterns is substantial; this remains true also when $\bar{B}(t)$ alone is considered; this is explained below. 

In fact $\bar{B}(t)$ is the only quantity that can be verified experimentally without imaging.   Consider momentarily the unconstrained and constrained models governed by \eqref{eq:ssystem} and \eqref{eq:PVI}, respectively. With unlimited nutrient supply, the growth of $\bar{B}(t)$ solving \eqref{eq:ssystem} in a closed system without the constraint \eqref{eq:constraint} should be exponential. (See the simulation in the Appendix.) However, the experiments in \cite{CVTAKPT} demonstrate that the microbial growth within the biofilm phase \add{is not exponential except in the beginning when $\Omega^* \approx \emptyset$. Our simulation of \eqref{eq:PVI} confirms that  once the constraint is active and the set $\Omega^*$ is growing, the growth ceases to be exponential.}  These effects are stronger when we use small or nonlinear diffusivities. 

We explore these variants in simulations. Fig.~\ref{fig:totB_1D} shows the evolution of the total amount of biofilm $\bar{B}(t)$ depending on the variant (i) through (iv); \add{these particular models are entirely heuristic}. We see that the growth of $\bar{B}(t)$ before reaching the maximum density is exponential in the case of constant $D_B$  up to the time the free boundary starts forming. The growth is faster with larger $D_B$, but the free boundary moves faster with smaller $D_B$.
In the case of nonlinear diffusivity (iii) as in \cite{PTISW2016}, as well as in the more singular variant (iv)  similar to that used in phase field models, e.g., \cite{zhangklapper}, or hybrid models in \cite{Eberal}, we see again that the growth of $\bar{B}(t)$ is exponential up to when $\Omega^*$ forms, but a more singular diffusivity gives a more pronounced  tapering effect. \add{Here by ``more singular'' $D_B(B)$ we mean a function with a large gradient close to $B^*$. We do not use a truly singular  $D_B(B)$ which blows up at $B\uparrow B^*$ such as in \cite{Eberal}.}

\begin{figure}
    \centering
    \includegraphics[scale=0.3]{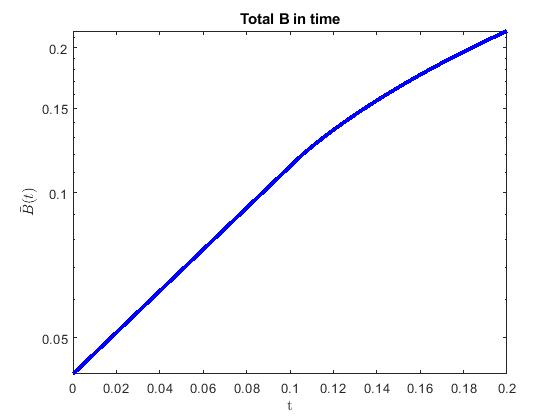}
    \includegraphics[scale=0.3]{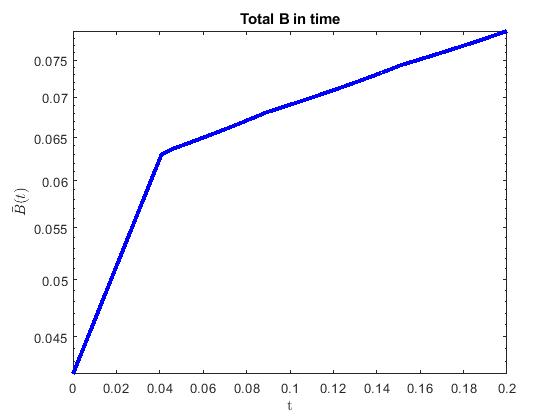}
\\
    \includegraphics[scale=0.3]{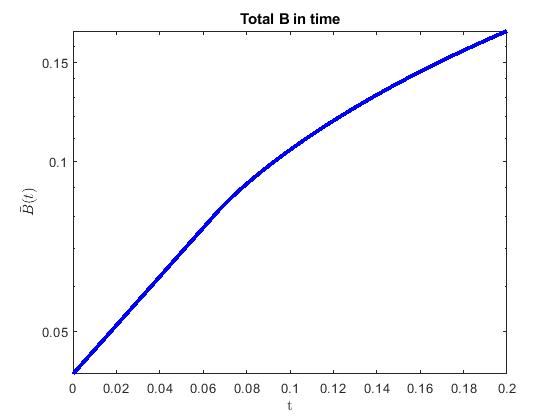}
    \includegraphics[scale=0.3]{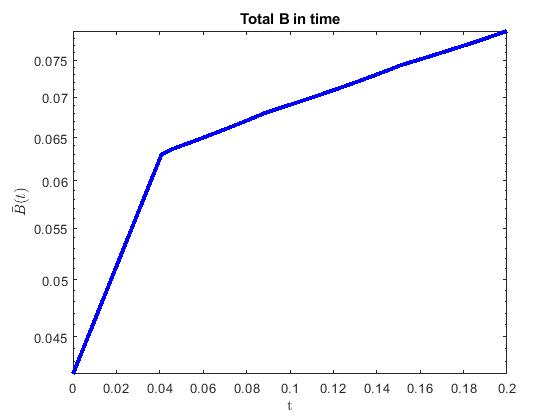}
    \caption{Total $\bar{B}(t)$ on log-scale in Ex.~\ref{ex:1DNeu} with  (i) $D_B=0.1$ (top left), (ii) $D_B=0.001$ (top right), 
    (iii) $D_B=(D_{\max}-D_{\min})(B/B^*)+D_{\min}$ (bottom left), (iv) $D_B=(D_{\max}-D_{\min})(B^8/B^*)+D_{\min}$ (bottom right);  $D_{\max}=0.1$, $D_{\min}=0.001$. The tapering of exponential growth is more pronounced for more singular diffusivity, \add{with a high gradient close to $B^*$}.    
    }
    \label{fig:totB_1D}
    \end{figure}

\subsection{Simulations in $d=2$}
In this section we verify the theoretical result of Theorem~\ref{the:convergence}. 
We also show that the behavior of the coupled biofilm-nutrient dynamics depends significantly on the geometry of the domain $\Omega$, the boundary conditions, and the model for diffusivities $D_B(B)$. The examples are designed to show the growth through interface, starting from an initial biomass concentrated in a disk 
$D(r)$ centered at the origin with radius $r$.   

\begin{example}[Simulation in $d=2$, with Dirichlet boundary conditions] \label{ex:2Ddiri}
Consider the square domain $\Omega=(-1,1)^2$ with $ D_B=0.01$, and $D_N=0.5$. We set $ B_{{init}}= 0.2\chi_{D(0.5)}$, and ${B^\ast}=0.3$. We also set $N_{init}= \chi_{D(0.75)}$.   The Monod functions are $F=\frac{5N}{N+0.7}B$, $G=-\frac{0.5N}{N+0.7}B$. 
\end{example}

%------------------------------------------------
\begin{figure}
    \centering
    \includegraphics[scale=0.3]{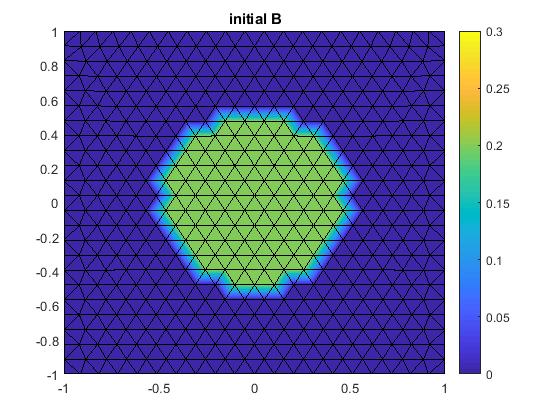}
    \includegraphics[scale=0.3]{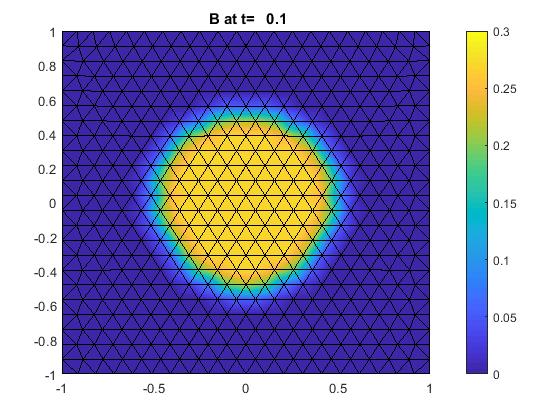}
%    \includegraphics[scale=0.41]{pic/B_ex2_T02.jpg}
%    \includegraphics[scale=0.41]{pic/B_ex2_T04.jpg}
%    \caption{Growth of $B$  in Example~\ref{ex:2Ddiri}  with Dirichlet boundary conditions and $h=0.1$\label{fig:B0}
 %   }
 %   \end{figure}
 %   \begin{figure}
   \includegraphics[scale=0.3]{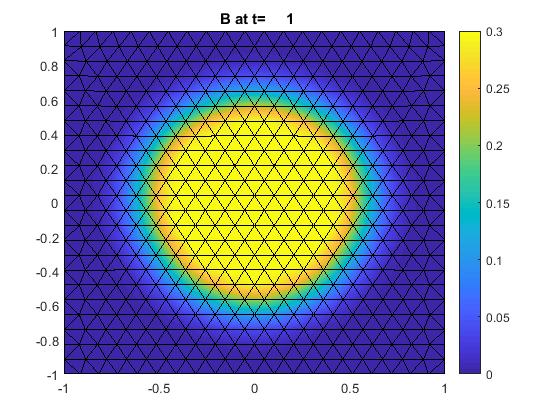}
    \includegraphics[scale=0.3]{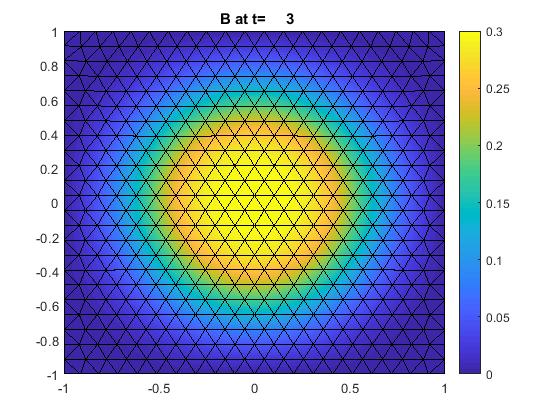}
    \caption{Evolution of $B(x,t)$ in Ex.~\ref{ex:2Ddiri}; Dirichlet bc.}
    \label{fig:B}
\end{figure}
%-----------------------------------------------
\begin{figure}
    %\centering
    \includegraphics[scale=0.3]{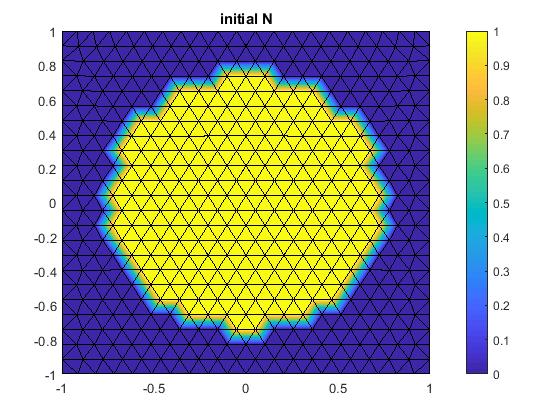}
    \includegraphics[scale=0.3]{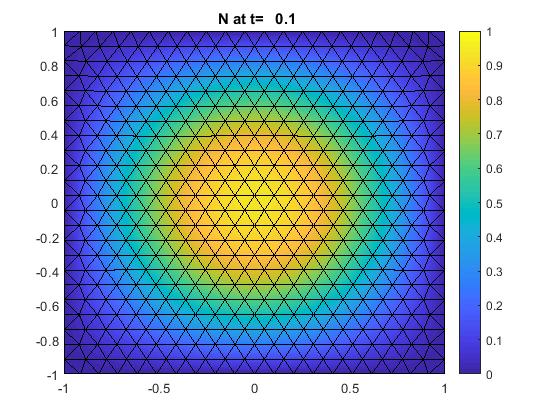}
    \includegraphics[scale=0.3]{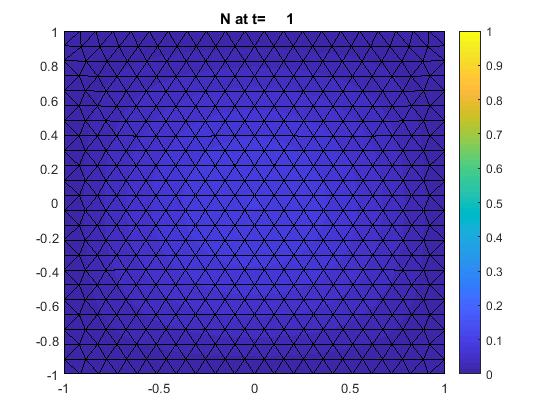}
    \includegraphics[scale=0.3]{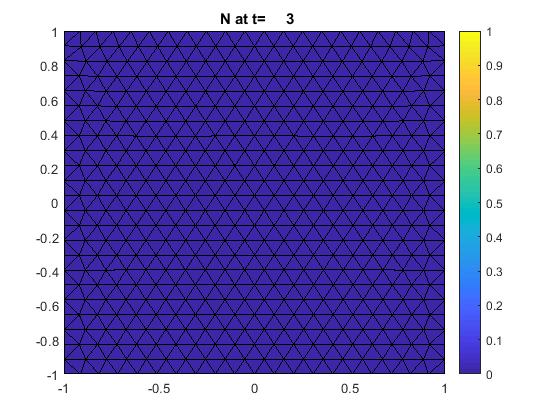}
    \caption{Consumption of $N(x,t)$ in Ex~\ref{ex:2Ddiri}; Dirichlet dc.}
    \label{fig:N}
\end{figure}
%-------------------------------------------------------------

%%%%%
\subsubsection{}
Figures~\ref{fig:B} and \ref{fig:N} show the evolution of $B$ and $N$ over time; here $h=0.1$. We use 494 nodes, and 899 elements. 

The growth of $B(x,t)$ for $0<t<0.4$ is vigorous and concentrated near its initial position, and then for $0.4 < t <1$ its spread through the interface. Nutrient is consumed most substantially where $B$ grows. Around $t=1$ both $B$ and $N$ start to decay, and the evolution is dominated by the escape of $B$ and $N$ through the boundary due to the homogeneous Dirichlet conditions.

%%%%%%
\subsubsection{} For convergence test, we compute the fine grid solution $(B_{fine},N_{fine})$ as a proxy for $(B,N)$, triangulating the domain with $\mathcal{T}_{fine} = 22887$ triangles, with 11660 nodes and 34546 edges, where the maximum length of each side of the triangles is $h_{fine}=0.02$. % 
We use $\Delta t_{fine}=5\times 10^{-5}$, and for the purposes of convergence testing we store
the numerical solution $(B_{fine},N_{fine})$  at $\Upsilon=\{0.1,0.2,0.3\}$. 

We compute the solution at coarse grids and compare it to $B_{fine}$ and $N_{fine}$, calculating the associated values of the approximation error. The errors are given in Table~\ref{tab:2D}. 
 It seems that $ERR1^\Upsilon$ is essentially of the first order, whereas $ERR2^\Upsilon$ is of $O(h^{3/2})$, similarly as in the case of $d=1$. Again we see that the free boundary is smoothly expanding, and the size of $\sum_n m(D_n)$ is only mildly increasing. 

%------------------------------------------------
%------------------------------%%T=0.1%%-----------------------------------
\begin{table}
    \centering
      \begin{tabular}{cccccccc}
\hline
$h$& $\Delta t$&$\#$nodes&  $\#$ elements&ERR1& ERR2& ERR1 ord.& ERR2 ord.\\
\hline
0.15&0.006&219&378&0.0499347&0.0474&&\\
0.1& 0.004&494& 899 &0.0282514&0.0274&1.4047&1.3517\\
0.05&0.002&1906&3638&0.0113125  &0.0087&1.3204&1.6551\\
\hline
\end{tabular}

\begin{tabular}{cccccccc}
\hline
$h$& $\Delta t$&$\#$nodes&  $\#$ elements&ERR1& ERR2& ERR1 ord.& ERR2 ord.\\
\hline 
0.15& 0.006& 219& 378& 0.0411& 0.0550&&\\
0.1& 0.004&494 &899 & 0.0221& 0.0371&1.5302&0.9710\\
0.05&0.002&  1906 &3638 & 0.0108&0.0141&1.0330&1.3957\\
\hline
\end{tabular}
\caption{Convergence test for Ex.~\ref{ex:2Ddiri}; Dirichlet bc (top)
%\label{tab:err2Ddir}
and  for Ex.~\ref{ex:2DNeu1}; Neumann bc. (bottom). We use $\Upsilon=\{0.1,0.2,0.3\}$.}
%\label{tab: 2DN}
\label{tab:2D}
\end{table}

%----------------------------------------------

%--------------------------------------------

\subsection{Experiments in $d=2$ with homogeneous Neumann boundary conditions }

    %%%%%%%%%%%%%%%%%%%%%%%%%%%%%
\begin{example}[2D Simulation] \label{ex:2DNeu1}
 In this example, we consider data as in Ex.~\ref{ex:2Ddiri} except the boundary and initial conditions. We choose homogeneous Neumann conditions to model an isolated system, a nonsymmetric initial condition 
 $ B_{init}=0.3\chi_{(-0.75,0)\times(-0.5,0.5)}$, and provide abundant nutrient $ N_{init}\equiv 1$. 
 \end{example}

We note that the case of Neumann boundary conditions is not covered by the theory. However, the errors  shown in Table~\ref{tab:2D} demonstrate that the  order of convergence is similar to that we obtained for Dirichlet boundary conditions.  

Furthermore, we track the cumulative growth, i.e., $\bar{B}(t)$; see Fig.~\ref{fig:totalB}, in which we compare the growth to that in more complex geometry and in $d$=$3$ addressed below. 

%%%

%
%-------------------------------------------------

%----------------------------------------------------

%-------------------------------------------------

%---------------------------------------------

%---------------------------------------------
\subsection{Simulation in complex \porescale\ geometry and in $d$=$3$.}
\begin{example}[Simulation in  porescale geometry]\label{ex:2DNeu3}
We use the data in Table \ref{tab:ex5}, and we consider a domain $\Omega$ with geometry motivated by \cite{PTISW2016}. We use nonlinear diffusivity $D_B=D_B(B)$, $F(B,N)=\kappa_B\frac{N}{N+\eta}B$, $G(B,N)=-\kappa_N\frac{N}{N+\eta}B$, $\Omega_b$ is the initial biofilm domain as shown in the illustration.  \end{example}
 
 The evolution of biofilm is shown in Fig.~\ref{fig:ex5BN} and the total amount $\bar{B}(t)$ in Fig.~\ref{fig:totalB}. We see the effects of reaching $B^*$ as well as of the limited propagation through interface due to the complex geometry. 

\begin{table}
    \centering
    \begin{tabular}{l l }
\hline
%Domain&$\Omega_l=B(0,1)\subset\mathbb{R}^3$\\
%Biofilm domain &$\Omega_b=\{(x,y,z); 0<x<0.5,$\\
%&$\;\;-0.25<y<1,\; -1<z<0.5\}$\\
Maximum  density of biomass& $B^\ast=0.12 kg/m^3$ \\
%of biomass &  \\
Growth constant & $\kappa_B= 1.8 /s$\\
Utilization constant & $\kappa_N= 1.8 \cdot 10/s$\\
Monod constant & $\eta=0.16 kg/m^3$\\
Nutrient diffusivity &$D_N(x)=20 m/s $ \\
Biomass diffusivity & $D_B(B)=(D^\ast-D_\ast)(B/B^\ast)+D_\ast$\\
& $D^\ast=0.01,\; D_\ast=10^{-4}D^\ast$\\
Initial nutrient&$N_{0}(x)=\chi_{\Omega}$\\
Initial biomass &$B_{0}(x)=0.03\chi_{\Omega_b}$\\
\hline
  \end{tabular}
    \caption{Data in Ex.~\ref{ex:2DNeu3} and \ref{ex:3D} from \cite{PTISW2016} scaled by $10^{5}$}
    \label{tab:ex5}
\end{table}

\begin{figure}
    \centering
    \includegraphics[scale=0.41]{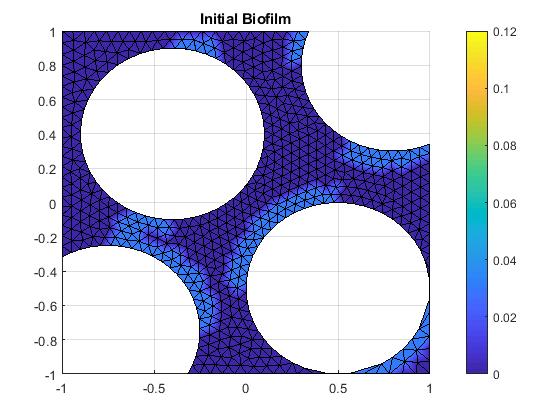}
    \includegraphics[scale=0.41]{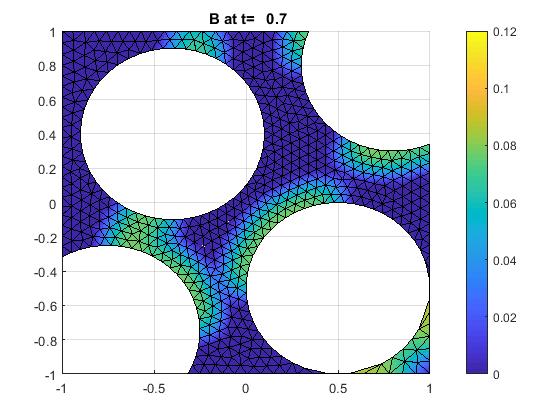}
    \includegraphics[scale=0.41]{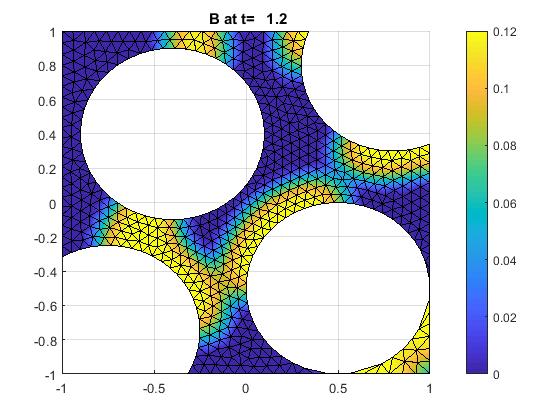}
    \includegraphics[scale=0.41]{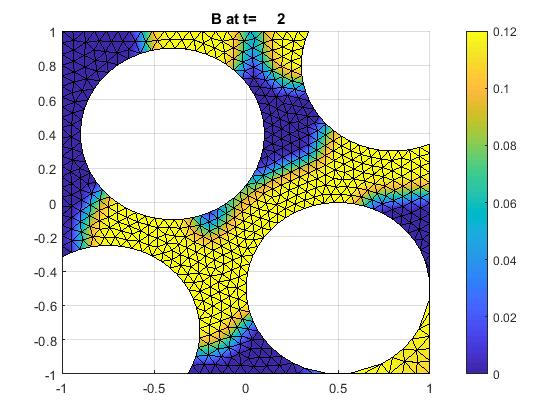}
    \caption{Growth of $B(x,t)$ in Ex.~\ref{ex:2DNeu3} simulated with $h=0.1$ \add{and $899$ elements}; Neumann bc.}
    \label{fig:ex5BN}
\end{figure}

%---------------------------------------------

\begin{figure}
    \centering
    \includegraphics[scale=0.25]{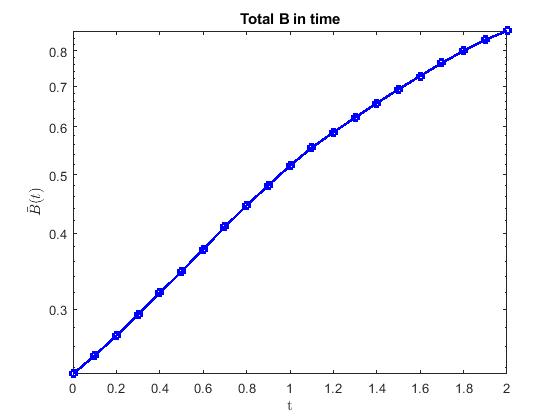}
   \includegraphics[scale=0.25]{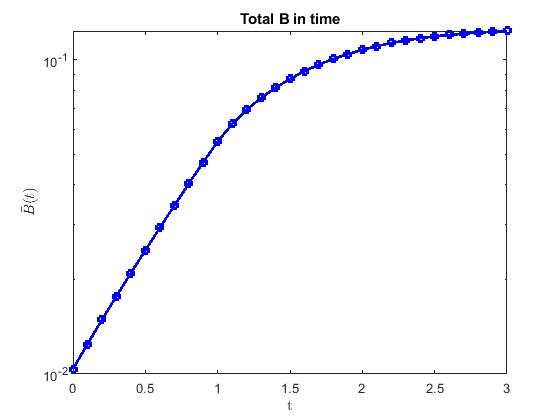}
\includegraphics[scale=0.25]{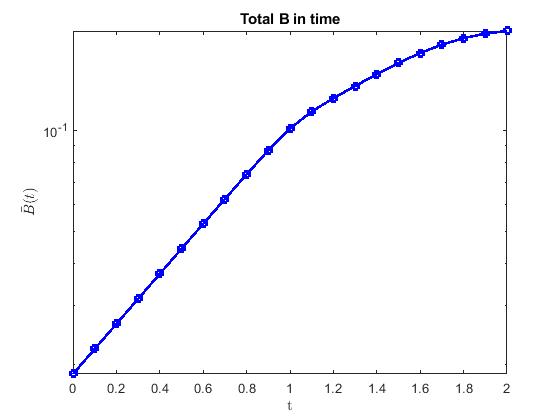}
\caption{Total $\bar{B}(t)$
for simulations with Neumann bc for Ex.~\ref{ex:2DNeu1} (simple $d$=$2$ geometry, porescale geometry in Ex.~\ref{ex:2DNeu3}, and in $d$=$3$ in Ex.~\ref{ex:3D}.}
    \label{fig:totalB}
\end{figure}

%---------------------------------------

%\subsection{Simulations in $d=3$}

\begin{example}[3D Simulation] \label{ex:3D}We use the data in Table \ref{tab:ex5}.
In this example, $\Omega$ is a ball with two holes which represent solid grains. The initial biofilm is chosen to adhere to the solid surfaces as is shown in Figure~\ref{fig:Bex6}. 
\end{example}

As time goes, biofilm keeps growing in its initial domain until $T\approx 1$ when it reaches its maximum $B^*=0.12$. After that,  biofilm starts spreading through the interface. The evolution of biofilm is illustrated in Fig.~\ref{fig:Bex6}.
%%%%%%%
\begin{figure}
\centering
    \includegraphics[scale=0.3]{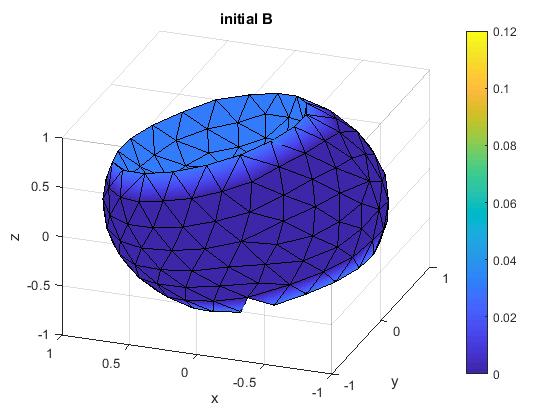}
        \includegraphics[scale=0.3]{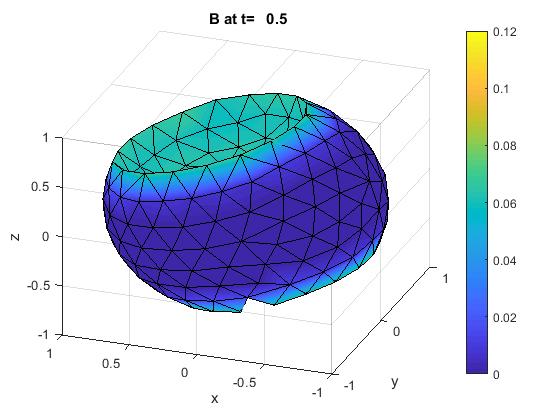}
        \includegraphics[scale=0.3]{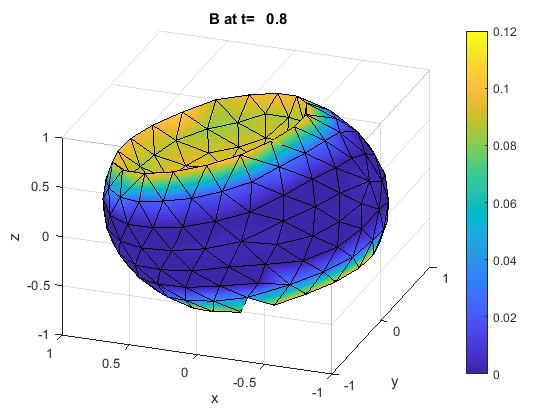}
        \includegraphics[scale=0.3]{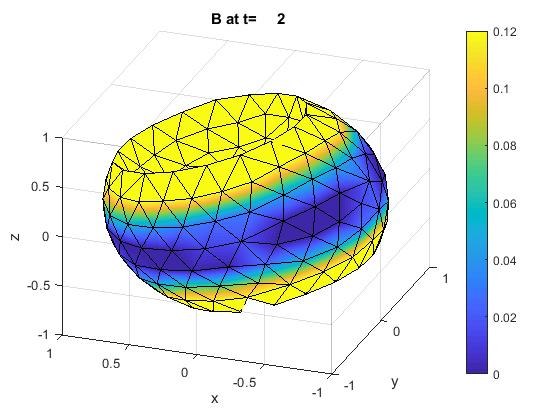}
        \caption{Evolution of biofilm in Ex.~\ref{ex:3D}, with {$h=0.2$} \add{and 1397 elements.}}
        \label{fig:Bex6}
    \end{figure}
    %-------------------------------------------

%%%

\medskip
We conclude with a comment on the behavior of $\bar{B}(t)$ 
shown in Fig.~\ref{fig:totalB}. The qualitative behavior for the Ex.~\ref{ex:2DNeu3} and \ref{ex:3D} seems different from than that in the generic bulk 2d domain of Ex.~\ref{ex:2DNeu1}. The growth of $\bar{B}(t)$ in Fig.~\ref{fig:totalB} appears initially fast in all cases, and a plateau of exponential growth is achieved about the time the growth through interface begins. We do not see this effect as strongly in Ex.~\ref{ex:2DNeu1} where constant diffusivity is used, but we see this effect pronounced in the other two examples. Furthermore, the growth through interface is further restricted in complex \porescale\ geometry. 

%%%%%%%%%%%%%%
\section{Conclusions and summary}

In this paper we proved and verified error estimates for a nonlinear coupled system of parabolic variational inequalities modeling biofilm--nutrient dynamics at microscale. 

We derived error estimates of rate $O(h+(log \Delta t^{-1})^{1/4}\Delta t^{3/4})$ in $l^\infty(L^2)$ and $l^2(H^1)$- norms. The rate of convergence was validated experimentally in 1D and 2D. Although the theoretical analysis dealt with Dirichlet boundary conditions only, the numerical results showed the same rate of convergence also with Neumann boundary conditions. Moreover, simulations in 2D and 3D with irregular geometries analogous to those obtained from imaging at porescale and with data  motivated by realistic simulations in \cite{PTISW2016} showed typical behavior of biofilm and nutrient in porous media. 

Further studies and analyses of the importance of modeling choice of $D_B(B), D_N(N)$ are underway. We are also considering rigorous analysis of the approximation \add{with advection and coupling to the flow}, and with other than Galerkin FE. Last but not least, \add{we hope to be able to address further properties of the solutions such as non-negativity, e.g., through some form of Discrete Maximum Principle. }

%%%%%%%%
\subsection*{Acknowledgements}
\add{First, we would like to thank the anonymous reviewers for their suggestions which helped to improve our paper.} 

Next, some of the work reported in this paper is part of the PhD thesis \cite{Azharthesis} by Dr Azhar Alhammali at Oregon State University. She would like to gratefully acknowledge the support of Saudi Arabian Cultural Mission to the US. In addition, 
this research was  partially supported by the NSF grant 
DMS-1522734 ``Phase transitions in porous media across multiple scales'', \add{and DMS-1912938 ``Modeling with Constraints and Phase Transitions in Porous Media'', as well as by M.~Peszynska's NSF IRD plan 2019-20. }

Last but not least we would like to thank the collaborators on \cite{PTISW2016} and in particular Dr Anna Trykozko for providing initial inspiration for the modeling efforts on this project. 

%%%%%%%%%%%%%%%%%%%%%%%  
\section*{Appendix} \label{apx}
Here we provide two examples which provide context for and can be compared to our convergence and simulations reported in Sec.\ref{sec:NE}. 

In particular, in Ex.~\ref{ex:solPVI} we show convergence of FE method for a scalar linear PVI; this should be compared with that for the coupled constrained system consider in Sec.~\ref{sec:1d}.  

%%%%%%%%
%--------------------------------------
%-------------------------------------------
\begin{example}\label{ex:solPVI}
In this example we approximate solutions to a scalar PVI
\begin{subequations}
\label{eq:exPVI}
\begin{align}
    B_t-0.5{ B_{xx}}+\partial_{[-0.04,  0.06]}B\ni& \pi^2\sin(x)B+3H(0.5-x)-3H(x-0.5) \;\mbox{on}\; (0,1),\; t>0,\\
    B(0,t)=0=&B(1,t),\; t>0,\\
    B(x,0)=&0.04\sin(\pi x)
\end{align}.
\end{subequations}

The approximation error in Fig.~\ref{fig:appendix} for this scalar problem shows convergence order similar to that in our 1d examples for the coupled system in Sec.~\ref{sec:1d}.
\end{example}

In turn, in Ex.~\ref{ex:Uncon} we show convergence rate   $O(\Delta t +h^2)$  of an unconstrained coupled system analyzed in \cite{Azharthesis}.

%----------------------------------
\begin{figure}
    \centering
\includegraphics[scale=0.3]{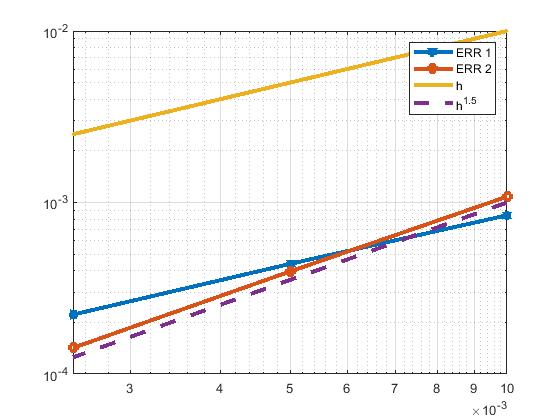}
\includegraphics[scale=0.3]{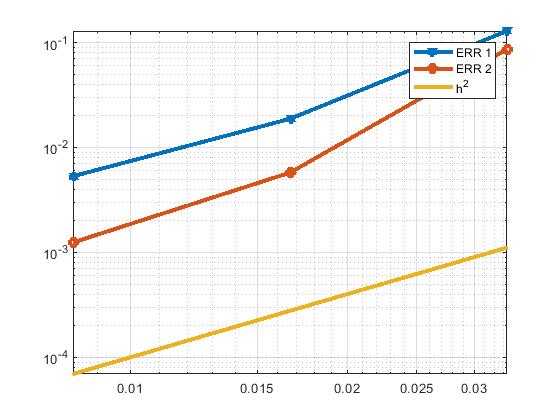}
\caption{Left: rate of convergence of FE solution for scalar PVI of Ex.~\ref{ex:solPVI}, $ERR 1=\max_n\|B^n-B_h^n\|_{ 0}$.
$ERR 2=\sqrt{\sum_n\|B^n-B_h^n\|_{ 1}^2 \Delta t}$.
Right: rate for convergence of Ex.~\ref{ex:Uncon}; here ERR1 and ERR2 are in a product space.}
    \label{fig:appendix}
\end{figure}

%%%%%%%%%%%%%%%%%%%
%\subsection{Rate of convergence of piecewise finite element approximation of unconstrained coupled system}\label{sec:ERRUnconst}
\begin{example}[Unconstrained Coupled System in 1D] \label{ex:Uncon}
{
 Consider the same data in Example \ref{ex:1D}, but $B$ in this case is unconstrained, i.e. $B^*=\infty$; see Fig.~\ref{fig:uncos}. The convergence error seems to be the usual  $O(h^2+\Delta t)$ in $ERR1=l^\infty(L^2)$ and  $ERR2=l^2(H_0^1)$ as in 
 Fig.~\ref{fig:appendix}.
 This should be compared with the system under constraints in Sec~\ref{sec:1d} with rate close to $O(\Delta t+h)$.
  }

\begin{figure}
    \centering
    \includegraphics[scale=0.3]{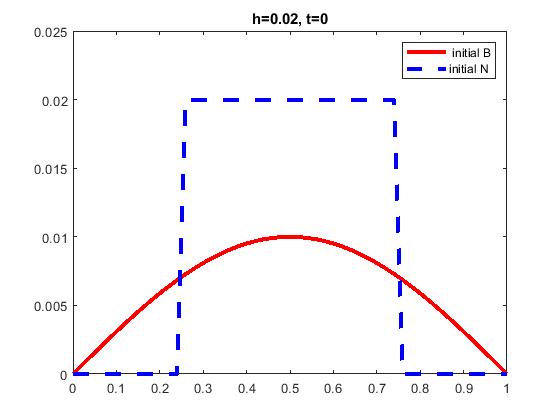}
    \includegraphics[scale=0.3]{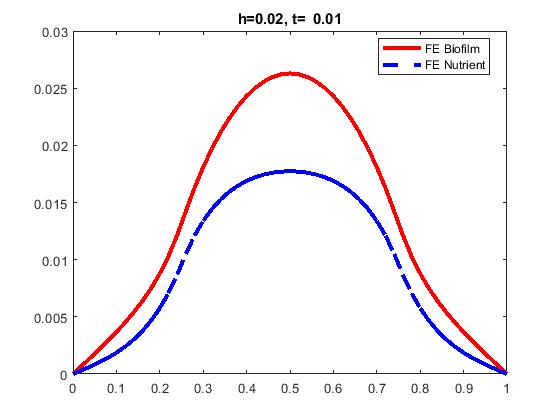}
    \includegraphics[scale=0.3]{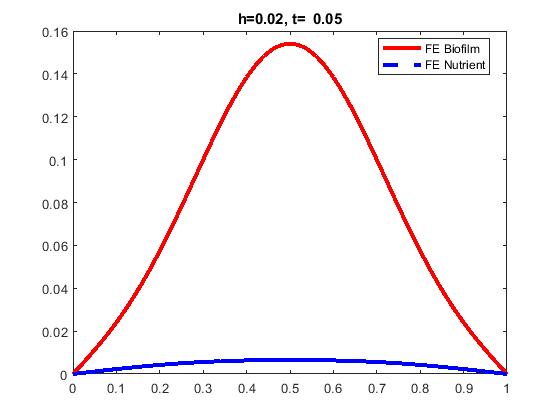}
    \includegraphics[scale=0.3]{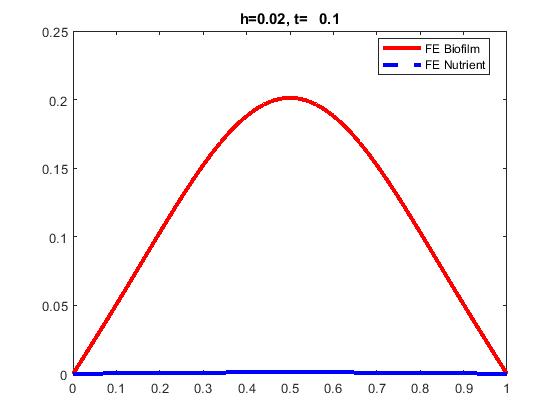}
    \caption{Solutions of Ex.~\ref{ex:Uncon} without constraints}
    \label{fig:uncos}
\end{figure}

\end{example}

%%%%%%%%%%%%%%%%%%%%%%%%%%%%%%%%%%%%%%%%%%%
\bibliography{Azharbib,mpesz}
\bibliographystyle{plain}
\end{document}